\documentclass[
final
 , nomarks
]{dmtcs-episciences}

\usepackage[utf8]{inputenc}
\usepackage{subfigure}

\usepackage{amssymb}
\usepackage{mathtools}

\usepackage{pgfplots}
\usepackage{tikz}
\usetikzlibrary{arrows,patterns,shapes,snakes,shadows,intersections}

\let\leq\leqslant
\let\geq\geqslant

\let\setminus\smallsetminus

\let\rho\varrho
\let\implies\Rightarrow
\let\iff\Leftrightarrow

\newcommand{\brac}[1]{{\left(#1\right)}}
\newcommand{\sbrac}[1]{{\left[#1\right]}}
\newcommand{\tbrac}[1]{{\left<#1\right>}}
\newcommand{\set}[1]{\left\{#1\right\}}
\newcommand{\norm}[1]{{\left|#1\right|}}

\newcommand{\Oh}[1]{O\brac{#1}}

\newtheorem{theorem}{Theorem}

\newtheorem{corollary}[theorem]{Corollary}
\newtheorem{lemma}[theorem]{Lemma}
\newtheorem{observation}[theorem]{Observation}

\usepackage[round]{natbib}

\author{Patryk Mikos\affiliationmark{1}\thanks{Research was partially supported by the National Science Center of Poland under grant no.~2014/14/A/ST6/00138.}}
\title{Efficient Enumeration of Non-isomorphic Interval Graphs}
\affiliation{
  Theoretical Computer Science Department,
  Faculty of Mathematics and Computer Science,
  Jagiellonian University, Krak\'ow, Poland}
\keywords{combinatorics, graph theory, enumeration, interval graphs}

\received{2020-02-27}
\revised{2020-09-15}
\accepted{2020-12-23}

\begin{document}
\publicationdetails{23}{2021}{1}{2}{6164}

\maketitle

\begin{abstract}
Recently, Yamazaki et al. provided an algorithm that enumerates all
non-isomorphic interval graphs on $n$ vertices with an $\Oh{n^4}$ time delay.
In this paper, we improve their algorithm and achieve $\Oh{n^3 \log n}$ time delay.
We also extend the catalog of these graphs providing a list of all non-isomorphic interval graphs for all $n$ up to $15$.
\end{abstract}

\section{Introduction}
\label{sec:Intro}

Graph enumeration problems, besides their theoretical value, are of interest not only for computer scientists,
but also to other fields, such as physics, chemistry, or biology.
Enumeration is helpful when we want to verify some hypothesis on a quite big set of different instances, or find a small counterexample.
For graphs it is natural to say that two graphs are "different" if they are non-isomorphic.
Many papers dealing with the problem of enumeration were published for certain graph classes, see \cite{Kiyomi06,Saitoh10,Saitoh12,Yamazaki18}.
A series of potential applications in molecular biology, DNA sequencing, network multiplexing, resource allocation, job scheduling,
and many other problems, makes the class of \emph{interval graphs}, i.e. intersection graphs of intervals on the real line,
a particularly interesting class of graphs.
In this paper, we focus on interval graphs,
and our goal is to find an efficient algorithm that for a given $n$ lists all non-isomorphic interval graphs on $n$ vertices.
It is well-known that the number of such graphs is roughly $n^{nc}$ for some constant $c$, see \cite{Huseyin18,Halon82,Yang17}.
For that reason, we measure the efficiency of the enumeration algorithm by the worst-case time delay between output of any two successive graphs.

\subsection{Previous work}
\cite{Yamazaki18} presented an enumeration algorithm for non-isomorphic interval graphs
that works with the worst-case time delay $\Oh{n^6}$, and recently \cite{Yamazaki19} improved it to $\Oh{n^4}$.
Their algorithm is based on the \emph{reverse search method}, invented by \cite{Avis96}, and in its general form works in the following way.
Let $\mathcal{C}$ be a family of graphs we want to enumerate, and let $G_1,\ldots,G_k \in \mathcal{C}$ be some special graphs called \emph{roots}.
We define a \emph{family forest} $\mathcal{F}$ which spans $\mathcal{C}$, and consists of $k$ rooted trees $\mathcal{T}_1,\ldots,\mathcal{T}_k$ such that the root of $\mathcal{T}_i$ is $G_i$.
For graphs that are not roots, let $par: \mathcal{C} \setminus \set{G_1,\ldots,G_k} \rightarrow \mathcal{C}$ be the \emph{parent function}.
In order to enumerate all graphs in the family $\mathcal{C}$, we enumerate all graphs that belong to each tree independently,
and to enumerate all graphs in the tree $\mathcal{T}_i$ we use any tree traversal algorithm like BFS or DFS.
The most time consuming operation in the tree traversal is computing the children of a graph $G$.
From the definition, children of $G$ are those graphs $G' \in \mathcal{C}$ whose parent is $G$.
Hence, if we want to design a fast enumeration algorithm that uses this technique, we need to carefully define the parent function.
\cite{Yamazaki19} used the fact that for every interval graph $G = \brac{V,E}$ that is not a complete graph,
there is at least one edge $e$ such that the graph $G_e = \brac{V,E \cup \set{e}}$ is also an interval graph.
They defined the only root $G_1$ to be a complete graph on $n$ vertices,
and $par\brac{G} = G_e$, where the edge $e$ is uniquely defined - for more details see \cite{Yamazaki19}.
The consequence of this approach is the fact that for every graph $G$ on $n$ vertices there are at most $\Oh{n^2}$ candidates for the children of $G$ in $\mathcal{F}$,
and so this number does not depend on the size of the enumerated family.
Moreover, the authors observed that in order to enumerate only non-isomorphic graphs, it is enough to filter out isomorphic copies from the set of children.
Isomorphism test in the class of interval graphs is an easy problem thanks to \emph{MPQ}-trees, see Section \ref{sec:Preliminaries}.
Hence, to compute the set of children for a graph $G = \brac{V,E}$, the authors consider all graphs $G_e = \brac{V,E \setminus \set{e}}$.
For each of them, they check whether $G_e$ is an interval graph using some linear time recognition algorithm.
Then, if $G_e$ is an interval graph, they check whether $G = par\brac{G_e}$, and build a corresponding \emph{MPQ}-tree.
Finally, they store a set of children trees, effectively removing all duplicates.

\subsection{Our results}
In this paper we revisit the work of Yamazaki et al. and show how to modify their enumeration algorithm
to significantly reduce the worst-case time delay between the output of two successive graphs.
Our key observation is the fact that having an \emph{MPQ}-tree corresponding to a graph $G = \brac{V,E}$
we are able to list all edges $e$ such that a graph $G_e = \brac{V,E \setminus \set{e}}$ is an interval graph.
Moreover, for each such edge we show how to build an \emph{MPQ}-tree corresponding to the graph $G_e$
without constructing it explicitly.

\subsection{Organization of this paper}
In the next section we introduce concepts and definitions that are widely used in this paper,
and also provide a detailed description of \emph{MPQ}-trees along with their most important properties.
In Section \ref{sec:CanonicalMPQ}, we present a total ordering over all \emph{MPQ}-trees, define a canonical \emph{MPQ}-tree using this ordering,
and also present a fast algorithm that for a given \emph{MPQ}-tree $\mathcal{T}$ computes its canonical form $\mathcal{T}'$.
In Section \ref{sec:IntervalEdges} we consider an \emph{MPQ}-tree $\mathcal{T}$ corresponding to an interval graph $G = \brac{V,E}$,
and characterize edges $e$ such that the graph $G_e = \brac{V,E \setminus \set{e}}$ is also an interval graph.
Moreover, for every edge $e$ we either show a linear time algorithm that produces a string representing $G_e$ if it is an interval graph,
or show an induced chordless cycle on four vertices or an asteroidal triple in $G_e$ that certifies that $G_e$ is not an interval graph.
In Section \ref{sec:Listing} we develop data structures and algorithms that make use of combinatorial characterization from Section \ref{sec:IntervalEdges}
and present a fast algorithm which for a given \emph{MPQ}-tree lists all edges $e$ such that $G_e$ is an interval graph.
Finally, in Section \ref{sec:ParentChild} we show how to combine all parts together and build the graph enumeration algorithm.
We also show the worst-case performance analysis of our algorithm in this section.
The last section contains a discussion of some implementation heuristics that do not change the worst-case analysis,
but significantly speedup the execution.

\section{Preliminaries}
\label{sec:Preliminaries}

In this paper we consider only simple graphs without loops and multiple edges.
We use the standard notations for graphs, so $n = \norm{V\brac{G}}$ and $m = \norm{E\brac{G}}$.
For a graph $G = \brac{V,E}$ and a pair of vertices $i,j \in V$,
we denote $G + \brac{i,j}$ a graph $G' = \brac{V,E \cup \set{\brac{i,j}}}$,
and $G - \brac{i,j}$ a graph $G' = \brac{V,E \setminus \set{\brac{i,j}}}$.
A graph $G = \brac{V,E}$ with a vertex set $V = \set{1,\ldots,n}$ is an \emph{interval graph}
if there is a set of intervals $\mathcal{I} = \set{I_1,\ldots,I_n}$ on the real line
such that $\brac{i,j} \in E$ iff $I_i \cap I_j \neq \emptyset$.
The set $\mathcal{I}$ is called an \emph{interval representation} of the graph $G$.
For an interval graph $G$, we say that an edge $\brac{i,j} \in E\brac{G}$ is an \emph{interval edge}
if $G - \brac{i,j}$ is also an interval graph.
A sequence $\mathcal{S}$ of length $2n$ is called a \emph{string representation}
if each element $x \in \set{1,\ldots,n}$ appears exactly two times in $\mathcal{S}$. 
Note that a string representation $\mathcal{S}$ encodes an interval graph in a natural way.
For every $x \in \set{1,\ldots,n}$ let $first\brac{x}$ denote the index of the first appearance of $x$ in $\mathcal{S}$, $second\brac{x}$ denote the second one, and $x$ is represented by an interval $I_x = \sbrac{first\brac{x},second\brac{x}}$.

\subsection{\emph{PQ}-trees}

It is easy to notice that an interval graph can have many different interval representations.
\cite{Lueker79} introduced a data structure, called a \emph{PQ}-tree,
which encodes all normalized interval representations of an interval graph.
A \emph{PQ}-tree is a rooted labeled plane tree composed of leaves and two kinds of internal nodes called \emph{P}-nodes, and \emph{Q}-nodes respectively.
The left to right ordering of the leaves of a \emph{PQ}-tree $T$ is called the \emph{frontier} of $T$.
We say that $T$ encodes an interval graph $G$,
if each maximal clique of the graph $G$ is stored in exactly one leaf of $T$,
and each vertex $v \in V\brac{G}$ belongs to a consecutive sequence of cliques in the frontier of $T$.
Having a \emph{PQ}-tree $T$ one can obtain another \emph{PQ}-tree $T'$ which is \emph{equivalent} to $T$
using the following two operations: arbitrarily permute the children of a \emph{P}-node, or reverse the order of the children of a \emph{Q}-node.
The crucial property of a \emph{PQ}-tree $T$ is the fact that for every permutation $\sigma$ of maximal cliques of the graph $G$
such that each vertex belongs to a consecutive sequence of cliques, there is a \emph{PQ}-tree $T'$ that is equivalent to $T$, and frontier of $T'$ represents $\sigma$.
In other words, each normalized interval representation of the graph $G$ is represented by some tree equivalent to $T$.

\subsection{\emph{MPQ}-trees}

\emph{PQ}-trees are quite a simple and easy to understand data structure representing interval graphs, but unfortunately they may occupy up to $\Oh{n^2}$ space.
To reduce the space consumption, \cite{Korte89} presented \emph{modified} \emph{PQ}-trees called \emph{MPQ}-trees.
In an \emph{MPQ}-tree, we do not store maximal cliques in leaves,
but we assign to each \emph{P}-node and each child of a \emph{Q}-node a set of vertices in such a way
that vertices laying on a path from the root of the tree to some leaf represent a maximal clique in $G$,
see Figure \ref{fig:MpqTree}C for an example.
For a \emph{Q}-node $Q$ with children $T_1,\ldots,T_k$,
we denote $S_i$ the set of vertices assigned to $T_i$, and call it the \emph{i-th section} of $Q$.
Note that, each vertex belongs to the consecutive sequence of maximal cliques, so it has to belong to consecutive sequence of sections of a \emph{Q}-node.
Hence, in order to limit the used space, we can store the information about the vertex $x$ only in the first and last section it belongs to.
Thanks to this modification, an \emph{MPQ}-tree is an $\Oh{n}$ space representation of an interval graph.
In this paper we show several drawings of \emph{MPQ}-trees.
We represent \emph{P}-nodes as circles, and \emph{Q}-nodes as rectangles divided into smaller rectangles representing sections of the \emph{Q}-node.
For instance, in the Figure \ref{fig:MpqTree}C the root is an empty \emph{P}-node, and the vertex $6$ belongs to the sections $S_2$ and $S_3$ of the only \emph{Q}-node.

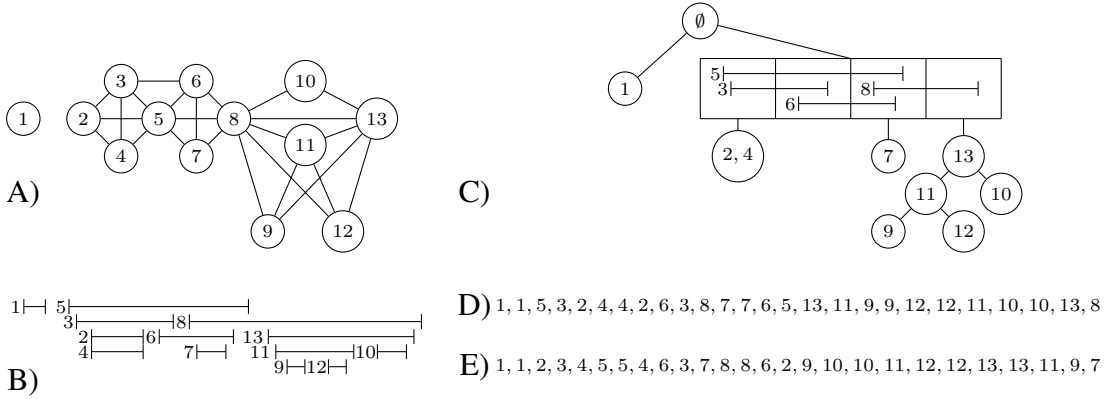
\begin{figure}[h]
\begin{center}
\begin{scriptsize}
\begin{tikzpicture}

    \node at (0,-1) {\large A)};

    \node[draw,circle] (n1) at (0,0) {$1$};
    \node[draw,circle] (n2) at (1.3,0.5) {$3$};
    \node[draw,circle] (n3) at (0.8,0) {$2$};
    \node[draw,circle] (n4) at (1.3,-0.5) {$4$};
    \node[draw,circle] (n5) at (1.8,0) {$5$};
    \node[draw,circle] (n6) at (2.3,0.5) {$6$};
    \node[draw,circle] (n7) at (2.8,0) {$8$};
    \node[draw,circle] (n8) at (2.3,-0.5){$7$};
    \node[draw,circle] (n9) at (4.7,0) {$13$};
    \node[draw,circle] (n10) at (3.75,-0.35) {$11$};
    \node[draw,circle] (n11) at (3.25,-1.5) {$9$};
    \node[draw,circle] (n12) at (4.25,-1.5) {$12$};
    \node[draw,circle] (n13) at (3.75,0.5) {$10$};

    \draw [-] (n3) -- (n4) -- (n5) -- (n3) -- (n2) -- (n5) -- (n6) -- (n2) -- (n4);
    \draw [-] (n5) -- (n8) -- (n7) -- (n5);
    \draw [-] (n6) -- (n8);
    \draw [-] (n6) -- (n7) -- (n9) -- (n13) -- (n7);
    \draw [-] (n7) -- (n11) -- (n10) -- (n7) -- (n12) -- (n10) -- (n9);
    \draw [-] (n12) -- (n9) -- (n11);

    \node at (0,-3.5) {\large B)};

    \draw [|-|] (0,-2.5) -- (0.3,-2.5); \node at (-0.1,-2.5) {$1$};
    \draw [|-|] (0.6,-2.5) -- (3,-2.5); \node at (0.5,-2.5) {$5$};
    \draw [|-|] (2.2,-2.7) -- (5.3,-2.7); \node at (2.1,-2.7) {$8$};
    \draw [|-|] (1.8,-2.9) -- (2.8,-2.9); \node at (1.7,-2.9) {$6$};
    \draw [|-|] (0.7,-2.7) -- (2,-2.7); \node at (0.6,-2.7) {$3$};
    \draw [|-|] (0.9,-2.9) -- (1.6,-2.9); \node at (0.8,-2.9) {$2$};
    \draw [|-|] (0.9,-3.1) -- (1.6,-3.1); \node at (0.8,-3.1) {$4$};
    \draw [|-|] (2.3,-3.1) -- (2.7,-3.1); \node at (2.2,-3.1) {$7$};
    \draw [|-|] (3.25,-2.9) -- (5.2,-2.9); \node at (3.05,-2.9) {$13$};
    \draw [|-|] (3.35,-3.1) -- (4.4,-3.1); \node at (3.15,-3.1) {$11$};
    \draw [|-|] (4.7,-3.1) -- (5.1,-3.1); \node at (4.55,-3.1) {$10$};
    \draw [|-|] (3.5,-3.3) -- (3.75,-3.3); \node at (3.4,-3.3) {$9$};
    \draw [|-|] (4.05,-3.3) -- (4.3,-3.3); \node at (3.9,-3.3) {$12$};
    
    \node at (6,-1) {\large C)};

    \node[draw,circle] (troot) at (9,1.3) {$\emptyset$};
    \node[draw,circle] (tisolated) at (8,0.4) {$1$};
    
    \node[draw,circle] (tpair) at (9.5,-0.5) {$2,4$};
    \node[draw,circle] (ts3) at (11.5,-0.5) {$7$};
    \node[draw,circle] (ts4) at (12.5,-0.5) {$13$};
    \node[draw,circle] (ts4l) at (12,-1) {$11$};
    \node[draw,circle] (ts4r) at (13,-1) {$10$};
    \node[draw,circle] (ts4ll) at (11.5,-1.5) {$9$};
    \node[draw,circle] (ts4lr) at (12.5,-1.5) {$12$};

    \draw [-] (9,0) -- (13,0);
    \draw [-] (9,0.8) -- (13,0.8);
    \draw [-] (9,0) -- (9,0.8);
    \draw [-] (10,0) -- (10,0.8);
    \draw [-] (11,0) -- (11,0.8);
    \draw [-] (12,0) -- (12,0.8);
    \draw [-] (13,0) -- (13,0.8);

    \draw [|-|] (9.3,0.6) -- (11.7,0.6); \node at (9.2,0.6) {$5$};
    \draw [|-|] (9.4,0.4) -- (10.7,0.4); \node at (9.3,0.4) {$3$};
    \draw [|-|] (10.3,0.2) -- (11.6,0.2); \node at (10.2,0.2) {$6$};
    \draw [|-|] (11.3,0.4) -- (12.7,0.4); \node at (11.2,0.4) {$8$};

    \draw [-] (troot) -- (tisolated);
    \draw [-] (troot) -- (11,0.8);
    \draw [-] (9.5,0) -- (tpair);
    \draw [-] (11.5,0) -- (ts3);
    \draw [-] (12.5,0) -- (ts4);
    \draw [-] (ts4) -- (ts4l);
    \draw [-] (ts4) -- (ts4r);
    \draw [-] (ts4l) -- (ts4ll);
    \draw [-] (ts4l) -- (ts4lr);

    \node at (6,-2.5) {\large D)};
    \node at (10.3,-2.5) {$1,1,5,3,2,4,4,2,6,3,8,7,7,6,5,13,11,9,9,12,12,11,10,10,13,8$};
    \node at (6,-3.3) {\large E)};
    \node at (10.3,-3.3) {$1,1,2,3,4,5,5,4,6,3,7,8,8,6,2,9,10,10,11,12,12,13,13,11,9,7$};

\end{tikzpicture}
\end{scriptsize}
\end{center}
\caption{
    A) An interval graph $G$, 
    B) Its interval representation $\mathcal{I}$, 
    C) Its \emph{MPQ}-tree $\mathcal{T}$, 
    D) String representation $\mathcal{S}$ of tree $\mathcal{T}$,
    E) Canonical string representation
}
\label{fig:MpqTree}
\end{figure}

\subsection{Known results}

During past decades, many researchers published their results on constructing both \emph{PQ}-trees and \emph{MPQ}-trees.
Those trees were mostly used to determine whether a given graph $G = \brac{V,E}$ is an interval graph or not.
\cite{Lueker79} in their recognition algorithm used \emph{PQ}-trees and proved that for a given graph $G$
the corresponding \emph{PQ}-tree can be computed in $\Oh{n+m}$ time.
\cite{Korte89} presented analogous result for \emph{MPQ}-trees.
In this paper we are most interested in work of \cite{Saitoh07} who presented an algorithm
that constructs an \emph{MPQ}-tree for a given interval graph representation
and works in $\Oh{n \log n}$ time, or $\Oh{n}$ if the endpoints of intervals are given in an ascending order.

\begin{theorem}[\cite{Saitoh07} Thm.12]
\label{thm:LinearMPQTree}
If the graph $G$ is given as an interval representation such that the endpoints are sorted by the coordinates,
then there is an algorithm that produces an \emph{MPQ}-tree corresponding to $G$ in $\Oh{n}$ time.
\end{theorem}

Clearly, having a string representation of the graph $G$,
we can produce an interval representation satisfying the conditions of Theorem \ref{thm:LinearMPQTree} in $\Oh{n}$ time.
Hence, we have the following corollary.

\begin{corollary}
There is an algorithm that for a given string representation $\mathcal{S}$ of the graph $G$
builds a corresponding \emph{MPQ}-tree $\mathcal{T}$ in $\Oh{n}$ time.
\end{corollary}


Before we proceed to technical definitions and lemmas, we provide some naming conventions we are going to use in the rest of this paper.
To avoid a confusion when talking about elements of a graph and elements of a tree,
we always refer elements of a graph as \emph{vertices} and elements of a tree as \emph{nodes}.
For a vertex $v$ of a graph $G$, we denote $node\brac{v}$ the node of a corresponding \emph{MPQ}-tree $\mathcal{T}$
such that $v$ belongs to the set assigned to that node.
For a node with $k$ subtrees $T_1,\ldots,T_k$, we denote $V_i$ the set of all vertices that are assigned to the nodes of a subtree $T_i$.
If $V_i = \emptyset$, then we say that a subtree $T_i$ is empty.
For a \emph{Q}-node we say that a vertex $v$ \emph{has its left endpoint} in a section $S_{l\brac{v}}$, if $v$ belongs to $S_{l\brac{v}}$ and does not belong to any other section $S_b$ with $b < l\brac{v}$.
Analogously, we say that $v$ \emph{has its right endpoint} in $S_{r\brac{v}}$, if $v$ belongs to $S_{r\brac{v}}$ and does not belong to any other section $S_b$ with $b > r\brac{v}$.
Vertex $v$ is \emph{contained} in sections $S_a,\ldots,S_b$, if $a \leq l\brac{v} < r\brac{v} \leq b$.

For an \emph{MPQ}-tree $\mathcal{T}$ we define a $2n$-element string $\mathcal{S}$ called \emph{string representation} of $\mathcal{T}$.
This string is built recursively over the structure of $\mathcal{T}$.
For a \emph{P}-node we first output all vertices that belong to that node,
then recursively string representations of the children from left to right,
and at the end yet again all vertices that belong to that node, but now in the reversed order.
Hence, the string representation for a \emph{P}-node with vertices $\set{1,\ldots,k}$ and no children
is $123\ldots\brac{k-1}kk\brac{k-1}\ldots 321$.
A string representation for a \emph{Q}-node is a concatenation of string representations for its sections.
The string for a section $S_i$ starts with vertices that have its left endpoint in $S_i$,
then there is a string for a subtree $T_i$, and finally vertices that have its right endpoint in $S_i$.
It is easy to see that string representation of $\mathcal{T}$ is also a string representation of the graph corresponding to $\mathcal{T}$.
We also define a \emph{normalized string representation} of $\mathcal{T}$.
Consider a permutation $\sigma: \set{1,\ldots,n} \rightarrow \set{1,\ldots,n}$, and a string $\sigma\brac{\mathcal{S}}$,
which results from the application of $\sigma$ to each element of $\mathcal{S}$.
Normalized string representation is the lexicographically smallest string $\sigma\brac{\mathcal{S}}$ among all permutations $\sigma$.
Finally, we recall some properties of \emph{MPQ}-trees produced by the Algorithm from Theorem \ref{thm:LinearMPQTree}.
\begin{lemma}[\cite{Korte89,Uehara05}]
    \label{lemma:QNodeProperties}
    In the \emph{MPQ}-tree constructed in Theorem \ref{thm:LinearMPQTree}
    for every \emph{Q}-node with $k$ children we have:
    \begin{enumerate}[a)]
        \item $V_1 \neq \emptyset$ and $V_k \neq \emptyset$,
        \item $S_1 \subset S_2$ and $S_k \subset S_{k-1}$,
        \item $S_{i-1} \cap S_i \neq \emptyset$ for $2 \leq i \leq k$,
        \item $S_{i-1} \neq S_i$ for $2 \leq i \leq k$,
        \item $\brac{S_i \cap S_{i+1}} \setminus S_1 \neq \emptyset$ and $\brac{S_{i-1} \cap S_i} \setminus S_k \neq \emptyset$ for $2 \leq i \leq k-1$, and
        \item $\brac{S_{i-1} \cup V_{i-1}} \setminus S_i \neq \emptyset$ and $\brac{S_i \cup V_i} \setminus S_{i-1} \neq \emptyset$ for $2 \leq i \leq k$.
    \end{enumerate}
    Moreover:
    \begin{enumerate}[A)]
	\item[g)] there are no two empty \emph{P}-nodes such that one of them is a parent of the other,
	\item[h)] there is no \emph{P}-node that have only one child which root is also a \emph{P}-node, and
	\item[i)] \emph{P}-nodes have no empty children
    \end{enumerate}
    see Figure \ref{fig:MpqTreePEmptyPath}.
\end{lemma}

\begin{figure}[h]
\begin{center}
\begin{scriptsize}
\begin{tikzpicture}


    \node[draw,circle] (ln1) at (1,1.5) {\small $\emptyset$};
    \node[draw,circle] (ln2) at (0.5,0.75) {\small $\emptyset$};
    \node (l1) at (0,0) {\small $T_1$};
    \node (l2) at (1,0) {\small $T_2$};
    \node (l3) at (1.5,0.75) {\small $T_3$};

    \node[draw,circle] (rn) at (3.7,1.15) {\small $\emptyset$};
    \node (r1) at (3.1,0.35) {\small $T_1$};
    \node (r2) at (3.7,0.35) {\small $T_2$};
    \node (r3) at (4.3,0.35) {\small $T_3$};

    \node at (2.3,0.75) {\small $\rightarrow$};

    \draw [-] (ln1) -- (ln2);
    \draw [-] (ln1) -- (l3);
    \draw [-] (ln2) -- (l1);
    \draw [-] (ln2) -- (l2);
    
    \draw [-] (rn) -- (r1);
    \draw [-] (rn) -- (r2);
    \draw [-] (rn) -- (r3);

    \node[draw,circle] (p1) at (7,1.4) {$P_1$};
    \node[draw,circle] (p2) at (7,0.65) {$P_2$};
    \node (n1) at (6.5,0) {\small $T_1$};
    \node (ndots) at (7,0) {\small \ldots};
    \node (n2) at (7.5,0) {\small $T_k$};

    \draw [-] (p1) -- (p2);
    \draw [-] (p2) -- (n1);
    \draw [-] (p2) -- (n2);

    \node at (8,0.75) {\small $\rightarrow$};

    \node[draw,circle] (q) at (9,1) {$P_1+P_2$};
    \node (m1) at (8.5,0) {\small $T_1$};
    \node (ndots) at (9,0) {\small \ldots};
    \node (m2) at (9.5,0) {\small $T_k$};

    \draw [-] (q) -- (m1);
    \draw [-] (q) -- (m2);

\end{tikzpicture}
\end{scriptsize}
\end{center}
    \caption{\emph{MPQ}-trees do not contain two consecutive empty \emph{P}-nodes (left),
    or a \emph{P}-node with only one child which root is also a \emph{P}-node (right).}
\label{fig:MpqTreePEmptyPath}
\end{figure}
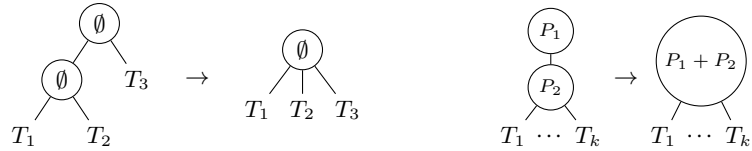

It is worth noting that \emph{MPQ}-trees describe interval graph in a natural recursive way.
If $G$ consists of at least $2$ connected components,
then the root node of \emph{MPQ}-tree $\mathcal{T}$ corresponding with $G$ is an empty \emph{P}-node, and each subtree corresponds to each connected component of $G$.
If $G$ is connected an vertices $v_{i_1},\ldots,v_{i_k}$ are universal in $G$ (vertex is universal if it is connected to all other vertices),
then the root of $\mathcal{T}$ node is a \emph{P}-node containing $v_{i_1},\ldots,v_{i_k}$.
In the remaining cases the root of $\mathcal{T}$ is a \emph{Q}-node.

\section{Canonical \emph{MPQ}-tree}
\label{sec:CanonicalMPQ}

In this section, we define a total ordering $\prec$ on \emph{MPQ}-trees.
One may notice that the lexicographical order on string representations is a total ordering on \emph{MPQ}-trees,
but for complexity reasons we introduce a different one.
Denote $\norm{\mathcal{T}}$ the number of vertices contained in the tree $\mathcal{T}$,
$c\brac{\mathcal{T}}$ the number of children of the root of $\mathcal{T}$,
and $ex\brac{\mathcal{T}}$ the number of vertices that belong to the root of $\mathcal{T}$.
We assign a tuple $t_{\mathcal{T}} = \tbrac{\norm{\mathcal{T}},ex\brac{\mathcal{T}},c\brac{\mathcal{T}}} \in \mathbb{N}^{3}$ to every tree $\mathcal{T}$,
and say that if $t_{\mathcal{T}_{1}}$ is lexicographically smaller than $t_{\mathcal{T}_2}$, then $\mathcal{T}_1 \prec \mathcal{T}_2$.
For two trees $\mathcal{T}_1$ and $\mathcal{T}_2$ such that $t_{\mathcal{T}_1} = t_{\mathcal{T}_2}$,
we say that $\mathcal{T}_1 \prec \mathcal{T}_2$, if the normalized string representation of $\mathcal{T}_1$ is lexicographically not greater
than normalized string representation  of $\mathcal{T}_2$.
We say that \emph{MPQ}-tree $\mathcal{T}$ is \emph{in canonical form},
if for every other tree $\mathcal{T}'$ representing the same graph $G$ we have $\mathcal{T} \prec \mathcal{T}'$.
A string $\mathcal{S}$ is a \emph{canonical string}, if it is a normalized string representation of a canonical tree.
Observe that if $\mathcal{T}$ is in a canonical form, then all subtrees of $\mathcal{T}$ are in a canonical form.
Clearly, if some subtree of $\mathcal{T}$ is not in a canonical form, then we may rotate it and obtain a lexicographically smaller string.

\begin{theorem}
\label{thm:TreeIsomorphism}
Two interval graphs $G_1$ and $G_2$ are isomorphic if and only if their canonical strings $S_1$ and $S_2$ are equal.
\end{theorem}

\begin{theorem}
\label{thm:MPQTreeCanonicalForm}
There is an algorithm that for every \emph{MPQ}-tree $\mathcal{T}$ computes its canonical form in $\Oh{n \log n}$ time.
\end{theorem}

\begin{proof}
At the very beginning, we shall compute a function $g$ that for every vertex $v$ will describe its relative position among all vertices from $node\brac{v}$.
We compute this function for each \emph{P}-node independently, and for all \emph{Q}-nodes collectively.
For a \emph{P}-node with $j$ vertices $z_1,\ldots,z_j$, we assign $g\brac{z_i} = i$.
Thus, we can compute the function $g$ for all \emph{P}-nodes in $\Oh{n}$ time.
In order to compute this function for all \emph{Q}-nodes,
at first we assign a tuple $\tbrac{l\brac{v},r\brac{v}}$ to each vertex $v$ which belongs to some \emph{Q}-node.
Then we sort all tuples using radix sort algorithm, and visit vertices in the order determined by their tuples.
For each \emph{Q}-node we keep a local counter that starts with $1$ and increases each time we visit a vertex from this node.
Thus, because all vertices are from the set $\set{1,\ldots,n}$, and each \emph{Q}-node has a linear in terms of $n$ number of sections,
we compute this function for all vertices in $\Oh{n}$ time.

We shall construct a function $f$ that assigns an integer $f\brac{\mathcal{T}'} > n$ to every subtree $\mathcal{T}'$ of a tree $\mathcal{T}$
in such a way that $\mathcal{T}_1 \prec \mathcal{T}_2 \iff f\brac{\mathcal{T}_1} \leq f\brac{\mathcal{T}_2}$.
Simultaneously we will rotate subtrees so that they are in canonical form.
At first, we compute tuples $t_{\mathcal{T}'}$ for each subtree $\mathcal{T}'$ of a tree $\mathcal{T}$.
Clearly, it can be easily done in $\Oh{n}$ time.
Then, we sort the tuples lexicographically in $\Oh{n}$ using radix sort algorithm.
In the next phases, we inspect nodes of tree $\mathcal{T}$ that have the same tuple $\tbrac{\norm{\mathcal{T}'},ex\brac{\mathcal{T}'},c\brac{\mathcal{T}'}}$,
and we do it from the smallest tuples to the biggest ones.
Observe that, when we define the value of the function $f$ for $\mathcal{T}'$, the values for all subtrees of $\mathcal{T}'$ are already computed.

All subtrees of $\mathcal{T}'$ are in a canonical form,
so in order to compute a canonical form of $\mathcal{T}'$, we need to determine the order of its children.
If the root of $\mathcal{T}'$ is a \emph{P}-node, then we use integers $f\brac{\mathcal{T}_i}$ as keys for children,
and sort them in $\Oh{c \log c}$ time, where $c = c\brac{\mathcal{T}'}$.
If the root of $\mathcal{T}'$ is a \emph{Q}-node $Q$, then we may leave it in the form it is or reverse it.
To decide what to do, we compute a special string representation $\mathcal{S}^{*}$, which is similar to the string representation,
but for each vertex $v$ that belongs to $Q$ we put $g\brac{v}$ instead of $v$,
and instead of inserting the whole string for a subtree $\mathcal{T}_i$, we put a single number $f\brac{\mathcal{T}_i}$.
Hence, the produced string has length $2*ex\brac{\mathcal{T}'}+c\brac{\mathcal{T}'}$ and is produced in time proportional to its length.
We also produce similar string for a rotated node, and if that string is lexicographically smaller than the original one,
then we rotate $Q$. Otherwise, we do nothing.

We have just computed canonical forms for all subtrees with the same tuple.
For each of them we produce a special string, and sort those strings lexicographically.
Finally, we assign values from the set $\set{F+1,F+2,\ldots}$,
where $F$ is the maximal number assigned to trees with lexicographically smaller tuples (or $F = n$ if there are no smaller).
We assign those numbers according to the computed order giving the same value to the subtrees with the same special string representation,
and that finishes the algorithm description.

Now, we prove the algorithm works in the declared time.
As we mentioned before, the computation of the function $g$ is linear in time.
The same applies to the computation and sorting for the node tuples.
Sorting children of a \emph{P}-node with $c$ children takes $\Oh{c \log c}$.
Hence, because all \emph{P}-nodes cannot have more than $\Oh{n}$ children in total,
we conclude that sorting children for all \emph{P}-nodes takes no more than $\Oh{n \log n}$ time.
The length of a special string for a \emph{Q}-node with $j$ vertices and $k$ sections is $\Oh{j+k}$.
Thus, the total processing time for all \emph{Q}-nodes is linear in terms of $n$.

The only thing we have not counted yet is the time spent on sorting subtrees with the same tuple.
Note that for a tuple $\tbrac{s,e,c}$, each special string has length exactly $2e+c$.
Let $n_{sec}$ be the number of subtrees having a tuple $\tbrac{s,e,c}$.
Sorting process for those subtrees takes no more than $\Oh{\brac{e+c}n_{sec} \log n_{sec}}$.
Thus, all sortings together take $\Oh{\sum_{sec} \brac{e+c}n_{sec} \log n_{sec}}$.
Note that $n_{sec} \leq n$, so we only need to show that $\sum_{sec} \brac{e+c}n_{sec}$ is $\Oh{n}$.
But, clearly $\sum_{sec} en_{sec} = n$ since this sum counts vertices in all nodes.
Similarly, $\sum_{sec} cn_{sec}$ equals the number of edges in $\mathcal{T}$, and we are done.
\end{proof}

\section{Interval edges}
\label{sec:IntervalEdges}

In this section we present a series of lemmas that characterize the interval edges for the interval graph $G$.
Moreover, for each interval edge $\brac{x,y}$, we also present a linear in terms of $n$ algorithm
that produces a string representation for the interval graph $G - \brac{x,y}$.
For an edge $\brac{x,y}$ that is not an interval edge,
we prove the existence of an induced chordless cycle on four vertices or an asteroidal triple in $G - \brac{x,y}$.
The characterization does not use the mere graph $G$, but the corresponding \emph{MPQ}-tree $\mathcal{T}$ instead.

\medskip
First, let us introduce an useful definition.
We say that $x$ \emph{is over} $y$ in $\mathcal{T}$,
if $\brac{x,y} \in E\brac{G}$ and $node\brac{x}$ is the lowest common ancestor of $node\brac{x}$ and $node\brac{y}$ in $\mathcal{T}$.
Notice that, if there is an edge $\brac{x,y}$ in the graph $G$, then $x$ is over $y$, or $y$ is over $x$.
Now, we make an easy observation on interval edges.

\begin{observation}
\label{obs:InducedC4}
If there are at least two vertices $z_1$ and $z_2$ such that both $x$ and $y$ are over $z_1$ and $z_2$,
and there is no edge $\brac{z_1,z_2}$, then $\brac{x,y}$ is not an interval edge.
\end{observation}

\begin{proof}
Vertices $x$, $z_1$, $y$ and $z_2$ in that order form a cycle of length $4$.
We assumed that there is no edge between $z_1$ and $z_2$,
so if there is no edge between $x$ and $y$, then this cycle is chordless in $G - \brac{x,y}$, see Figure \ref{fig:InducedC4}.
Hence, $G - \brac{x,y}$ is not a chordal graph and so not an interval graph.
\end{proof}

\begin{figure}[h]
\begin{center}
\begin{scriptsize}
\begin{tikzpicture}

    \node (x) [shape=circle,draw=black] at (0.3,1) {\small $x$};
    \node (y) [shape=circle,draw=black] at (1.2,1) {\small $y$};
    \node (z1) [shape=circle,draw=black] at (0,0) {\small $z_1$};
    \node (z2) [shape=circle,draw=black] at (1.5,0) {\small $z_2$};

    \path[-,dotted] (x) edge (y);
    \path[-] (x) edge (z1);
    \path[-] (x) edge (z2);
    \path[-] (y) edge (z1);
    \path[-] (y) edge (z2);

    \draw [->] (1.9,0.5) -- (2.6,0.5);

    \node (X) [shape=circle,draw=black] at (4,0) {\small $x$};
    \node (Y) [shape=circle,draw=black] at (3,1) {\small $y$};
    \node (Z1) [shape=circle,draw=black] at (3,0) {\small $z_1$};
    \node (Z2) [shape=circle,draw=black] at (4,1) {\small $z_2$};

    \path[-] (X) edge (Z1);
    \path[-] (X) edge (Z2);
    \path[-] (Y) edge (Z1);
    \path[-] (Y) edge (Z2);

\end{tikzpicture}
\end{scriptsize}
\end{center}
\caption{An induced $C_4$ after removing edge $\brac{x,y}$.}
\label{fig:InducedC4}
\end{figure}
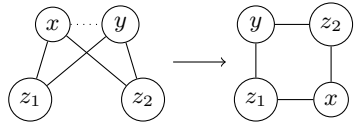

The above observation is useful when we want to prove that the edge $\brac{x,y}$ is not an interval edge.
However, in cases when $\brac{x,y}$ is an interval edge, we want to show a linear time algorithm that produces a string representation for a graph $G - \brac{x,y}$.
The following lemma, we call the \emph{swapping lemma},
comes handy when we try to produce the mentioned string.
It shows when we can swap two consecutive elements in a string representation without adding or removing any edges to the represented graph.

\begin{lemma}
\label{lemma:SwappingLemma}
Let $\mathcal{S}_1$ and $\mathcal{S}_2$ be string representations,
such that $\mathcal{S}_2$ is created from $\mathcal{S}_1$ by swapping elements at positions $i$ and $i+1$ for some $i$.
Denote $a$ the element at the $i$-th position in $\mathcal{S}_1$, and $b$ the element at the $\brac{i+1}$-th position ($\mathcal{S}_1 = ....ab..$ and $\mathcal{S}_2 = ....ba..$).
$\mathcal{S}_1$ and $\mathcal{S}_2$ represent the same interval graph iff both elements at swapped positions represent either left endpoints or right endpoints.
\end{lemma}

\begin{proof}
Clearly, at most one edge can be added or removed by swapping those two elements.
If we swap the left endpoint of $a$ and the right endpoint of $b$, then we remove an edge $\brac{a,b}$.
If we swap the right endpoint of $a$ and the left endpoint of $b$, then we add an edge $\brac{a,b}$ which is not present in $\mathcal{S}_1$.
If both elements represent left endpoints, then right endpoints for both $a$ and $b$ are to the right of $i+1$, hence no edge is added or removed.
Similar argument works when both elements represent right endpoints.
\end{proof}

Our aim is to characterize all interval edges encoded by an \emph{MPQ}-tree $\mathcal{T}$.
Hence, as an input we are given an \emph{MPQ}-tree $\mathcal{T}$ and some edge $\brac{x,y} \in E\brac{G}$.
Without loss of generality, we assume that $x$ is over $y$ in $\mathcal{T}$,
and we work under this assumption in the following subsections.
We split our argument into cases according to the relative position of $node\brac{x}$ and $node\brac{y}$ in $\mathcal{T}$.


\subsection{$x$ and $y$ belong to the same \emph{P}-node}

At first, we consider the case where both $x$ and $y$ belong to the same \emph{P}-node $P$ in $\mathcal{T}$.
We show that under this assumption, the edge $\brac{x,y}$ is an interval edge if and only if $P$ is a leaf in $\mathcal{T}$.

\begin{lemma}
\label{lemma:RemovingEdgePInternal}
If $node\brac{x} = node\brac{y}$ is a \emph{P}-node that is not a leaf,
then $\brac{x,y}$ is not an interval edge.
\end{lemma}

\begin{proof}
Let $P$ be the considered common \emph{P}-node, and assume it has $j$ subtrees for some $j \geq 1$.
If $j \geq 2$, then by Lemma \ref{lemma:QNodeProperties}i let $z_1 \in V_1$ and $z_2 \in V_2$.
Clearly, there is no edge between $z_1$ and $z_2$ and both $x$ and $y$ are over $z_1$ and $z_2$.
Hence, Observation \ref{obs:InducedC4} implies that $\brac{x,y}$ is not an interval edge.
Thus, $P$ has exactly one subtree, and according to Lemma \ref{lemma:QNodeProperties}h,
its root has to be a \emph{Q}-node, see Figure \ref{fig:InternalPNode}.
Moreover, Lemma \ref{lemma:QNodeProperties}a implies, that the first and last sections of a \emph{Q}-node have nonempty subtrees.
Let $z_1$ belong to the first subtree, and $z_2$ belong to the last one.
Yet again, conditions of the Observation \ref{obs:InducedC4} are satisfied, so $\brac{x,y}$ is not an interval edge.
\end{proof}

\begin{figure}[h]
\begin{center}
\begin{scriptsize}
\begin{tikzpicture}

    \node (p) [shape=circle,draw=black] at (1,1.8) {\small $x,y$};
    \draw [-] (1,1.4) -- (1,1);

    \draw [-] (0,0.5) -- (2,0.5);
    \draw [-] (0,1) -- (2,1);
    \draw [-] (0,0.5) -- (0,1);
    \draw [-] (0.5,0.5) -- (0.5,1);
    \draw [-] (1.5,0.5) -- (1.5,1);
    \draw [-] (2,0.5) -- (2,1);
    \node at (0.25, 0.75) {\small $S_1$};
    \node at (1,0.75) {\small \ldots};
    \node at (1.75, 0.75) {\small $S_k$};

    \draw [-] (0.25,0.2) -- (0.25,0.5);
    \draw [-] (1.75,0.2) -- (1.75,0.5);

    \node at (0.25,0) {\small $T_1$};
    \node at (1.75,0) {\small $T_k$};


    \node (T) [shape=circle,draw=black] at (4,1.8) {\small $x,y$};
    \node (T1) at (3.5,1) {\small $T_1$};
    \node (T2) at (4.5,1) {\small $T_2$};

    \path [-] (T) edge (T1);
    \path [-] (T) edge (T2);

\end{tikzpicture}
\end{scriptsize}
\end{center}
\caption{Removing an edge from a \emph{P}-node that is not a leaf leads to an induced $C_4$ cycle.}
\label{fig:InternalPNode}
\end{figure}
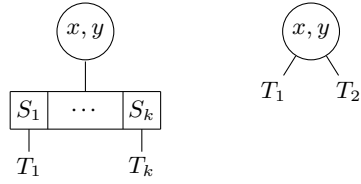

\begin{lemma}
\label{lemma:RemovingEdgePLeaf}
If $node\brac{x} = node\brac{y}$ is a \emph{P}-node that is a leaf, then $\brac{x,y}$ is an interval edge.
Moreover, there is a linear time algorithm that produces a string representation for the graph $G - \brac{x,y}$.
\end{lemma}

\begin{proof}
Without loss of generality assume that $x < y$, and consider the canonical string $\mathcal{S}$ for the \emph{MPQ}-tree $\mathcal{T}$.
Clearly, $\mathcal{S}$ is of the form $\mathcal{S} = LAxByC\bar{C}y\bar{B}x\bar{A}R$,
see Figure \ref{fig:StringPNodeInternal}.
In order to remove the edge $\brac{x,y}$ we find the first and the last occurrence of $x$ in $\mathcal{S}$.
Then, until $x$ does not occupy two consecutive positions, we swap $\mathcal{S}\sbrac{i}$ with $\mathcal{S}\sbrac{i+1}$,
and $\mathcal{S}\sbrac{j}$ with $\mathcal{S}\sbrac{j-1}$, where $i$ denotes the first occurrence of $x$ and $j$ denotes the second.
Next, we do the same for $y$, and as a result we get a string such that both $y$'s are next to each other and are surrounded by both $x$'s,
see Figure \ref{fig:StringPNodeInternal}c.
Finally, we swap the first occurrence of $y$ with the second occurrence of $x$, effectively removing the edge $\brac{x,y}$.
Clearly, this procedure runs in $\Oh{n}$ time, but we have to ensure that it does not add or remove any other edge.
Note that, all modifications are performed in a substring of $\mathcal{S}$ that represents a clique, and is of the form $z_{1}z_{2}\ldots z_{k}z_{k}\ldots z_{2}z_{1}$.
Hence, each swapping operation - except the last one - swapped either two left endpoints or two right endpoints.
Thus, Lemma \ref{lemma:SwappingLemma} ensures that no edge was added or removed during this process.
Finally, during the last swap $x$ and $y$ occupy four consecutive indexes.
Hence, the only affected vertices are $x$ and $y$.

\begin{figure}[h]
\begin{center}
\begin{small}
\begin{tikzpicture}

	\draw [decorate, decoration = {brace,amplitude=3}, xshift=0, yshift=-4] (-1.85,3.5) -- (-1.45,3.5) node [black, midway, yshift = 10] {\textit{prefix}};
	\draw [decorate, decoration = {brace,amplitude=5}, xshift=0, yshift=-4] (-1.45,3.5) -- (1.4,3.5) node [black, midway, yshift = 10] {\textit{PNode}};
	\draw [decorate, decoration = {brace,amplitude=3}, xshift=0, yshift=-4] (1.4,3.5) -- (1.85,3.5) node [black, midway, yshift = 10] {\textit{sufix}};

    \node at (0,3)   {\large $L|AxByC\bar{C}y\bar{B}x\bar{A}|R$};
    \node at (0,2.5) {\large $L|AByCxx\bar{C}y\bar{B}\bar{A}|R$};
    \node at (0,2)   {\large $L|ABCxyyx\bar{C}\bar{B}\bar{A}|R$};
    \node at (0,1.5) {\large $L|ABCxxyy\bar{C}\bar{B}\bar{A}|R$};

	\node at (-2.3,3)   {\textit{a)}};
	\node at (-2.3,2.5) {\textit{b)}};
	\node at (-2.3,2)   {\textit{c)}};
	\node at (-2.3,1.5) {\textit{d)}};

\end{tikzpicture}
\end{small}
\end{center}
\caption{Removing an edge from a leaf \emph{P}-node.}
\label{fig:StringPNodeInternal}
\end{figure}

\end{proof}


\subsection{$x$ and $y$ belong to the same \emph{Q}-node}
The next case is when both vertices belong to the same \emph{Q}-node.
Here we show that $\brac{x,y}$ is an interval edge if and only if $x$ and $y$ have exactly one common section
and the subtree of this section represents a clique (possibly empty).

\begin{lemma}
\label{lemma:RemovingEdgeQFail}
If $node\brac{x} = node\brac{y}$ is a \emph{Q}-node,
then $\brac{x,y}$ is not an interval edge if:
\begin{enumerate}
    \item $x$ and $y$ have more than one common section in $node\brac{x} = node\brac{y}$, or
    \item a subtree of the common section does not represent a clique (possibly empty).
\end{enumerate}
\end{lemma}

\begin{proof}
Assume that there is a common section $S_i$ and its nonempty subtree $T_i$ that does not represent a clique.
Hence, there are at least two vertices $z_1$ and $z_2$ in $V_i$ such that there is no edge between them, otherwise $T_i$ would represent a clique.
Thus, Observation \ref{obs:InducedC4} implies that $\brac{x,y}$ is not an interval edge.
Moreover, if there are two common sections $S_i$ and $S_j$ such that both of them have nonempty subtrees $T_i$ and $T_j$ respectively,
then we may choose $z_1 \in V_i$ and $z_2 \in V_j$ and use the same argument.
This proves that common sections have empty subtrees except at most one which represents a clique.
Now, assume that there is more than one common section ie. $S_i,S_{i+1},\ldots,S_{j-1},S_j$, and without loss of generality $S_i$ has an empty subtree.
Lemma \ref{lemma:QNodeProperties}f implies that there is a vertex $z_1$ for which the section $S_i$ is the last one
($z_1$ does not belong to sections $S_{i+1},\ldots,S_{j}$).
Notice that $z_1 \notin \set{x,y}$, otherwise $j=i$.
If there is a nonempty subtree $T_a$ for some $i < a \leq j$, then we choose $z_2 \in V_a$,
and Observation \ref{obs:InducedC4} leads to an induced chordless cycle.
Hence, all common subtrees are empty and Lemma \ref{lemma:QNodeProperties}f gives us a vertex $z_2$ for which $S_j$ is the first section.
Observation \ref{obs:InducedC4} for vertices $x,y,z_1$ and $z_2$ finishes the proof.
\end{proof}

\begin{lemma}
\label{lemma:RemovingEdgeQNode}
If $node\brac{x} = node\brac{y}$ is a \emph{Q}-node,
$x$ and $y$ have exactly one common section $S_i$ in it,
and the subtree $T_i$ represents a clique,
then $\brac{x,y}$ is an interval edge.
Moreover, there is a linear time algorithm that produces a string representation for the graph $G - \brac{x,y}$.
\end{lemma}

\begin{proof}
Again, we assume that $x < y$ and consider the canonical string $\mathcal{S}$ for the \emph{MPQ}-tree $\mathcal{T}$,
but in this case $\mathcal{S}$ has a more complex form than in the Lemma \ref{lemma:RemovingEdgePLeaf}.
In fact, it is of the form $P_{1}xP_{2}L_{1}yL_{2}V_{i}\bar{V_{i}}R_{1}xR_{2}S_{1}yS_{2}$,
where $L_{1} \cup L_{2}$ represents the left endpoints of vertices from $S_i$,
$R_{1} \cup R_{2}$ represents the right endpoints of vertices from $S_i$,
and $V_{i} \cup \bar{V_{i}}$ represents a clique from the subtree, see Figure \ref{fig:StringQNodeInternal}.
In order to remove the edge $\brac{x,y}$, at first we need to determine for each element in $\mathcal{S}$
whether it represents the left or the right endpoint.
It can be easily done in $\Oh{n}$, since all elements in $\mathcal{S}$ belong to the set $\set{1,\ldots,n}$.
The next phase swaps the first occurrence of $y$ with its successor until the next element represents a right endpoint.
Analogously, we swap the second occurrence of $x$ with its predecessor until the next element represents a left endpoint.
Clearly, because of Lemma \ref{lemma:SwappingLemma} we did not add or remove any edge till this moment,
and $\mathcal{S}$ looks like in Figure \ref{fig:StringQNodeInternal}c.
Finally, we can remove the edge $\brac{x,y}$ by swapping the first occurrence of $y$ with the second occurrence of $x$,
that in fact occupy consecutive positions in $\mathcal{S}$.

\begin{figure}[h]
\begin{center}
\begin{small}
\begin{tikzpicture}


	\draw [decorate, decoration = {brace,amplitude=5}, xshift=0, yshift=-4] (-2.75,3.5) -- (-1.5,3.5) node [black, midway, yshift = 10] {\textit{prefix}};
	\draw [decorate, decoration = {brace,amplitude=5}, xshift=0, yshift=-4] (-1.5,3.5) -- (1.5,3.5) node [black, midway, yshift = 10] {\textit{section $S_{i}$}};
	\draw [decorate, decoration = {brace,amplitude=5}, xshift=0, yshift=-4] (1.5,3.5) -- (2.75,3.5) node [black, midway, yshift = 10] {\textit{sufix}};

    \node at (0,3)   {\large $P_{1}xP_{2}|L_{1}yL_{2}V_{i}\bar{V_{i}}R_{1}xR_{2}|S_{1}yS_{2}$};
    \node at (0,2.5) {\large $P_{1}xP_{2}|L_{1}L_{2}V_{i}y\bar{V_{i}}R_{1}xR_{2}|S_{1}yS_{2}$};
    \node at (0,2)   {\large $P_{1}xP_{2}|L_{1}L_{2}V_{i}yx\bar{V_{i}}R_{1}R_{2}|S_{1}yS_{2}$};
    \node at (0,1.5) {\large $P_{1}xP_{2}|L_{1}L_{2}V_{i}xy\bar{V_{i}}R_{1}R_{2}|S_{1}yS_{2}$};

	\node at (-3,3)   {\textit{a)}};
	\node at (-3,2.5) {\textit{b)}};
	\node at (-3,2)   {\textit{c)}};
	\node at (-3,1.5) {\textit{d)}};


	\draw [-] (4.5,3) -- (8.5,3);
	\draw [-] (4.5,2) -- (8.5,2);

	\draw [-] (5.5,2) -- (5.5,3);
	\draw [-] (7.5,2) -- (7.5,3);

	\draw [-|] (4.7,2.8) -- (7,2.8);
	\draw [-|] (4.7,2.6) -- (6.6,2.6);
	\draw [|-] (6.4,2.4) -- (8.3,2.4);
	\draw [|-] (6,2.2) -- (8.3,2.2);
	\node at (7.2, 2.8) {$R$};
	\node at (6.8, 2.6) {$x$};
	\node at (6.2, 2.4) {$y$};
	\node at (5.8, 2.2) {$L$};

	\draw [-] (6.5,2) -- (6.5,1.6);
	\node at (6.5,1.2) {\large $T_i$};
	\node at (6.3,3.5) {\large $S_i$};
	\node at (5,2.3) {\ldots};
	\node at (8,2.7) {\ldots};

\end{tikzpicture}
\end{small}
\end{center}
\caption{Removing an edge from the same \emph{Q}-node.}
\label{fig:StringQNodeInternal}
\end{figure}
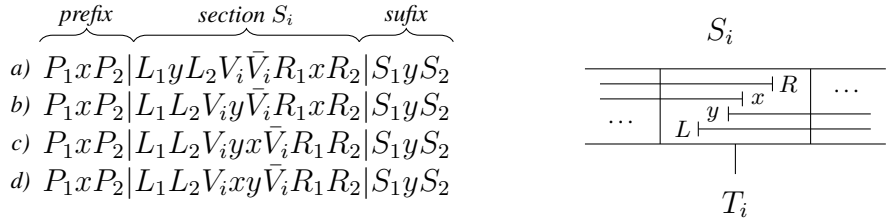

\end{proof}


\subsection{$x$ and $y$ belong to different nodes.}
In the previous two subsections, we provided a full classification for the cases where $x$ and $y$ belong to the same node in $\mathcal{T}$.
In this subsection, we consider the cases where $x$ and $y$ belong to different nodes in $\mathcal{T}$.
Before we present our results in those cases, we introduce a terminology that allows us to describe a relative position of $node\brac{x}$ and $node\brac{y}$ in $\mathcal{T}$.

For a \emph{Q}-node with $k$ sections $S_1,\ldots,S_k$,
we say that the section $S_{a}$ is a \emph{central} section if $1 < a < k$.
Sections $S_{1}$ and $S_{k}$ are called \emph{non-central} sections.
For a vertex $v$ 
we say that the section $S_a$ is a \emph{$v$-central} section if $l\brac{v} < a < r\brac{v}$.
Sections $S_{l\brac{v}}$ and $S_{r\brac{v}}$ are \emph{$v$-non-central} sections.
For every two vertices $x$ and $y$, there is exactly one path in the tree $\mathcal{T}$ between $node\brac{x}$ and $node\brac{y}$.
We say that this unique path is an \emph{$\tbrac{x,y}$-tree-path} if $x$ is over $y$ in $\mathcal{T}$.
For an $\tbrac{x,y}$-tree-path: $node\brac{x} = n_1 - n_2 - \ldots - n_t = node\brac{y}$,
we say that this path \emph{goes through the central section} if there is a \emph{Q}-node $n_i$ for $1 < i < t$ such that $n_{i+1}$ belongs to a subtree of some central section of $n_i$,
see Figure \ref{fig:PathThroughCentralSection}.
Moreover, we say that the path \emph{starts in a central section} if $n_1$ is a \emph{Q}-node, and $n_2$ belongs to a subtree of some $x$-central section.
Analogously, it \emph{starts in a non-central section} if $n_2$ belongs to a subtree of some $x$-non-central section.
We also say that an $\tbrac{x,y}$-tree-path \emph{starts} in a \emph{P}-node, if $node\brac{x}$ is a \emph{P}-node,
and \emph{ends} in a \emph{P}-node if $node\brac{y}$ is a \emph{P}-node.
Analogous definitions apply to \emph{Q}-nodes.
Finally, we say that an $\tbrac{x,y}$-tree-path is \emph{almost rotable} if it does not go through the central section, and ends in a \emph{P}-node that is a leaf.
An $\tbrac{x,y}$-tree-path is \emph{rotable} if it is almost rotable, and either starts in a \emph{P}-node, or starts in a non-central section.
Intuitively, if an $\tbrac{x,y}$-tree-path is almost rotable, then we are able to rotate all the nodes on the
path $n_2 - \ldots - n_t$ in such a way that $y$ is the leftmost vertex in the subtree which root is $n_2$.


\medskip
Now we are ready to characterize interval edges in the case where $x$ and $y$ belong to different nodes in $\mathcal{T}$.
We prove that if $\tbrac{x,y}$-tree-path is rotable, then $\brac{x,y}$ is an interval edge.
Unfortunately, the reverse implication is not true and sometimes $\tbrac{x,y}$-tree-path is not rotable, but $\brac{x,y}$ is still an interval edge.
We shall prove that this happens only for almost rotable $\tbrac{x,y}$-tree-paths satisfying some additional, and quite technical, conditions.
First, we show how to compute a string representation for an interval graph $G - \brac{x,y}$ if $\tbrac{x,y}$-tree-path is rotable.

\begin{lemma}
\label{lemma:RemoveEdgePPath}
If an $\tbrac{x,y}$-tree-path is rotable, then $\brac{x,y}$ is an interval edge.
Moreover, there is a linear time algorithm that produces a string representation for the graph $G - \brac{x,y}$.
\end{lemma}

\begin{proof}
In order to remove the edge $\brac{x,y}$, we do not produce a canonical string $\mathcal{S}$ immediately.
At first, we need to adjust the tree $\mathcal{T}$ using some preprocessing.
If $node\brac{x}$ is a \emph{Q}-node, then we rotate $node\brac{x}$ so that $\tbrac{x,y}$-tree-path starts in the section $S_{l\brac{x}}$.
Then, we rotate $\mathcal{T}$ so that the path from $node\brac{x}$ to $node\brac{y}$ goes through the leftmost children of \emph{P}-nodes and the leftmost sections of \emph{Q}-nodes.
Let $\mathcal{T}'$ be the result of the described adjustment, and let $\mathcal{S}$ be a string representation of $\mathcal{T}'$.
Clearly, $\mathcal{S}$ is of form $LxAByC\bar{C}y\bar{B}DxR$, see Figure \ref{fig:StringPNodePath}, and both occurrences of $y$ lay in between occurrences of $x$ in $\mathcal{S}$.
Node $node\brac{y}$ is a \emph{P}-node that is a leaf, so we start with moving the first occurrence of $y$ to the right and the second occurrence of $y$ to the left until they both meet, as in the Lemma \ref{lemma:RemovingEdgePLeaf}.
All vertices that lay between the first occurrence of $x$ and the first occurrence of $y$ represent the left endpoints.
Hence, we can swap the first occurrence of $x$ with its successor until it meets the second occurrence of $y$.
Lemma \ref{lemma:SwappingLemma} ensures that no edge is added or removed during this process.
Finally, moving $x$ once more to the right swaps the left endpoint of $x$ with the right endpoint of $y$ effectively removing the edge $\brac{x,y}$.

\begin{figure}[h]
\begin{center}
\begin{small}
\begin{tikzpicture}

	\draw [decorate, decoration = {brace,amplitude=3}, xshift=0, yshift=-4] (-1.8,3.5) -- (-1,3.5) node [black, midway, yshift = 10] {\textit{prefix}};
	\draw [decorate, decoration = {brace,amplitude=5}, xshift=0, yshift=-4] (-1,3.5) -- (0.9,3.5) node [black, midway, yshift = 10] {\textit{$node\brac{y}$}};
	\draw [decorate, decoration = {brace,amplitude=3}, xshift=0, yshift=-4] (0.9,3.5) -- (1.9,3.5) node [black, midway, yshift = 10] {\textit{sufix}};

	\draw [decorate, decoration = {brace,amplitude=3}, xshift=0, yshift=-4] (-1.8,1.8) -- (-1.2,1.8) node [black, midway, yshift = 10] {\textit{prefix}};
	\draw [decorate, decoration = {brace,amplitude=5}, xshift=0, yshift=-4] (-1.2,1.8) -- (0.9,1.8) node [black, midway, yshift = 10] {\textit{$node\brac{y} + x$}};
	\draw [decorate, decoration = {brace,amplitude=3}, xshift=0, yshift=-4] (0.9,1.8) -- (1.9,1.8) node [black, midway, yshift = 10] {\textit{sufix}};

	\node at (0,3)   {\large $LxA|ByC\bar{C}y\bar{B}|DxR$};
	\node at (0,2.5) {\large $LxA|BCyy\bar{C}\bar{B}|DxR$};

	\node at (0,1.3)   {\large $LA|BCyxy\bar{C}\bar{B}|DxR$};
	\node at (0,0.8) {\large $LA|BCyyx\bar{C}\bar{D}|DxR$};

	\node at (-2.3,3)   {\textit{a)}};
	\node at (-2.3,2.5) {\textit{b)}};
	\node at (-2.3,1.3) {\textit{c)}};
	\node at (-2.3,0.8) {\textit{d)}};

\end{tikzpicture}
\end{small}
\end{center}
\caption{Removing an edge $\brac{x,y}$ in the case where $\tbrac{x,y}$-tree-path is rotable.}
\label{fig:StringPNodePath}
\end{figure}

\end{proof}

Now, we consider those $\tbrac{x,y}$-tree-paths that are not rotable, but are almost rotable.
Hence, all those paths start in some $x$-central section of some \emph{Q}-node.
We denote $S_1,\ldots,S_k$ the sections of considered \emph{Q}-node, and $S_a$ the $x$-central section where the $\tbrac{x,y}$-tree-path starts.
As we already mentioned, sometimes in that case the edge $\brac{x,y}$ is an interval edge.
The next two lemmas establish the required conditions for that to happen.

\begin{lemma}
\label{lemma:RemoveEdgeInnerSectionNoNeighbour}
If an $\tbrac{x,y}$-tree-path is almost rotable, starts in a central section $S_{a}$, $y$ has no neighbor in a subtree $T_a$ and:
\begin{enumerate}
    \item $\exists_{1 < b \leq l\brac{x}}: S_{a} \setminus \set{x} \subset S_{b}$ and $S_{b-1} \cap S_{b} \subset S_{a}$, or
    \item $\exists_{r\brac{x} \leq b < k}: S_{a} \setminus \set{x} \subset S_{b}$ and $S_{b} \cap S_{b+1} \subset S_{a}$, or
    \item $S_{a} \setminus \set{x} \subset S_1$, or
    \item $S_{a} \setminus \set{x} \subset S_k$,
\end{enumerate}
then $\brac{x,y}$ is an interval edge.
Moreover, there is a linear time algorithm that produces a string representation for the graph $G - \brac{x,y}$.
\end{lemma}

\begin{proof}
Proofs for all four cases are very similar, so we show only the proof for the first case.
Assume that there is a $1 < b \leq l\brac{x}$ such that $S_{a} \setminus \set{x} \subset S_{b}$ and $S_{b-1} \cap S_{b} \subset S_{a}$.
As in the proof of Lemma \ref{lemma:RemoveEdgePPath}, before we produce the string representation $\mathcal{S}$,
we need to make some adjustments in the tree $\mathcal{T}$.
At first, we insert a new section $S^{*} = S_{a} \setminus \set{x}$ in between sections $S_{b-1}$ and $S_{b}$, see Figure \ref{fig:StringQPathNoEdge}.
Clearly, an insertion of a new section does not remove the old edges, but it might add some new ones.
However, the section $S^{*}$ is a subset of an already existing section, so it is not the case.
Moreover, the condition that each vertex belongs to the sequence of consecutive sections is preserved.
That's because we assumed that $S_{b-1} \cap S_{b} \subset S_{a}$ and $S_{a} \setminus \set{x} \subset S_{b}$.
Next, we remove the vertex $y$ from the subtree $T_a$.
We also define a subtree $T^{*}$ of the section $S^{*}$ to be a single \emph{P}-node containing $y$.
Note that $y$ has no neighbor in $T_a$, so no edges were removed, except the edge $\brac{x,y}$.
Thus, we obtained a tree $\mathcal{T}'$ that encodes the graph $G - \brac{x,y}$.
In order to get the string representation for the graph $G - \brac{x,y}$ it is enough to compute the string representation for $\mathcal{T}'$.
We proved only the first case, but observe that the second case is symmetric, and cases 3 and 4 are the corner cases, so in the third case we simply insert the section $S^{*}$ before $S_1$ and in the fourth case we insert $S^{*}$ after $S_k$.

\begin{figure}[h]
\begin{center}
\begin{scriptsize}
\begin{tikzpicture}

	\draw [-] (0,2) -- (6.9,2);
	\draw [-] (0,0.8) -- (6.9,0.8);

    \draw [-] (0.75,0.8) -- (0.75,2);
    \node at (1.1,2.3) {\small $S_{b-1}$};
    \draw [-] (1.5,0.8) -- (1.5,2);
    \node at (1.9,2.3) {\small $S_{b}$};
    \draw [-] (2.25,0.8) -- (2.25,2);
    \draw [-] (3.4,0.8) -- (3.4,2);
    \node at (3.75,2.3) {\small $S_{l\brac{x}}$};
    \draw [-] (4.15,0.8) -- (4.15,2);
    \draw [-] (5.4,0.8) -- (5.4,2);
    \node at (5.75,2.3) {\small $S_{a}$};
    \draw [-] (6.15,0.8) -- (6.15,2);

	\node at (5.75, 1.9) {$x$};
	\draw [|-] (3.6,1.8) -- (6.7,1.8);

    \draw [-|] (0.2,1.8) -- (1.3,1.8);
    \draw [|-|] (1.7,1.8) -- (3,1.8);
    \draw [|-] (1.8,1.6) -- (6.4,1.6);
    \draw [|-,dotted] (1,1.4) -- (6.4,1.4);
    \draw [-|,dotted] (0.2,1.2) -- (6,1.2);
    \draw [|-,dotted] (1.7,1) -- (6.4,1);

    \draw [-] (1.1,0.8) -- (1.1,0.6);
    \node at (1.1,0.35) {\small $T_{b-1}$};
    \draw [-] (1.9,0.8) -- (1.9,0.6);
    \node at (1.9,0.35) {\small $T_{b}$};
    \draw [-] (3.75,0.8) -- (3.75,0.6);
    \node at (3.75,0.35) {\small $T_{l\brac{x}}$};
    \draw [-] (5.75,0.8) -- (5.75,0.6);
    \node at (5.75,0.35) {\small $T_{a}$};

	\draw [-] (8,2) -- (15.65,2);
	\draw [-] (8,0.8) -- (15.65,0.8);

    \draw [-] (8.75,0.8) -- (8.75,2);
    \node at (9.1,2.3) {\small $S_{b-1}$};
    \draw [-] (9.5,0.8) -- (9.5,2);
    \node at (9.9,2.3) {\small $S^{*}$};
    \draw [-] (10.25,0.8) -- (10.25,2);
    \node at (10.65,2.3) {\small $S_{b}$};
    \draw [-] (11,0.8) -- (11,2);
    \draw [-] (12.15,0.8) -- (12.15,2);
    \node at (12.5,2.3) {\small $S_{l\brac{x}}$};
    \draw [-] (12.9,0.8) -- (12.9,2);
    \draw [-] (14.15,0.8) -- (14.15,2);
    \node at (14.5,2.3) {\small $S_{a}$};
    \draw [-] (14.9,0.8) -- (14.9,2);

	\node at (14.5, 1.9) {$x$};
	\draw [|-] (12.4,1.8) -- (15.45,1.8);

    \draw [-|] (8.2,1.8) -- (9.3,1.8);
    \draw [|-|] (10.45,1.8) -- (11.75,1.8);
    \draw [|-] (10.75,1.6) -- (15.15,1.6);
    \draw [|-,dotted] (9,1.4) -- (15.15,1.4);
    \draw [-|,dotted] (8.2,1.2) -- (14.75,1.2);
    \draw [|-,dotted] (9.7,1) -- (15.15,1);

    \draw [-] (9.1,0.8) -- (9.1,0.6);
    \node at (9.1,0.35) {\small $T_{b-1}$};
    \draw [-] (9.9,0.8) -- (9.9,0.6);
    \node [draw,circle] at (9.9,0.35) {\small $y$};
    \draw [-] (10.65,0.8) -- (10.65,0.6);
    \node at (10.65,0.35) {\small $T_{b}$};
    \draw [-] (12.5,0.8) -- (12.5,0.6);
    \node at (12.5,0.35) {\small $T_{l\brac{x}}$};
    \draw [-] (14.5,0.8) -- (14.5,0.6);
    \node at (14.5,0.35) {\small $T_{a} - y$};

    \node at (7.45,1.4) {\small $\rightarrow$};

\end{tikzpicture}
\end{scriptsize}
\end{center}
\caption{Removing an edge $\brac{x,y}$ in the case where $\tbrac{x,y}$-tree-path is almost rotable, starts in some $x$-central section $S_a$ and $y$ does not have a neighbor in a subtree $T_a$. (Case 1)}
\label{fig:StringQPathNoEdge}
\end{figure}
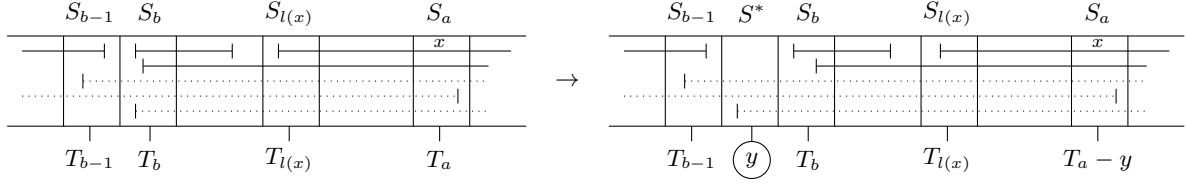

\end{proof}


\begin{lemma}
\label{lemma:RemoveEdgeInnerSectionNeighbour}
If an $\tbrac{x,y}$-tree-path is almost rotable, starts in a central section $S_{a}$, $y$ has a neighbor in a subtree $T_a$ and:
\begin{enumerate}
 \item $l\brac{x} > 1$ and $S_{a} \setminus \set{x} \subset S_{l\brac{x}}$ and $S_{l\brac{x}-1} \cap S_{l\brac{x}} \subset S_{a}$, or
 \item $r\brac{x} < k$ and $S_{a} \setminus \set{x} \subset S_{r\brac{x}}$ and $S_{r\brac{x}} \cap S_{r\brac{x}+1} \subset S_{a}$, or
 \item $l\brac{x} = 1$ and $S_{a} \setminus \set{x} \subset S_1$, or
 \item $r\brac{x} = k$ and $S_{a} \setminus \set{x} \subset S_k$,
\end{enumerate}
then $\brac{x,y}$ is an interval edge.
Moreover, there is a linear time algorithm that produces a string representation for the graph $G - \brac{x,y}$.
\end{lemma}

\begin{proof}
The proof of this lemma is similar to the proof of Lemma \ref{lemma:RemoveEdgeInnerSectionNoNeighbour},
but now $y$ has at least one neighbor in $T_a$, so we cannot simply remove $y$ from $T_a$.
That's also the reason why conditions of this lemma are more strict than in the previous one.
Yet again, we are going to prove only the first case,
so assume that $S_{a} \setminus \set{x} \subset S_{l\brac{x}}$ and $S_{l\brac{x}-1} \cap S_{l\brac{x}} \subset S_{a}$.
Instead of inserting a new section, we remove the section $S_{a}$ and insert it in between sections $S_{l\brac{x}-1}$ and $S_{l\brac{x}}$, see Figure \ref{fig:StringQPathEdge}.
Clearly, no edge is added or removed, and the condition that each vertex belongs to the sequence of consecutive sections is preserved.
Now, we remove $x$ from the section $S_{a}$.
This, removes the edge $\brac{x,y}$, but also all the edges $\brac{x,v}$, where $v \in V_{a}$.
In the next phase, we are going to restore those edges.
In order to do it, at first we rotate the subtree $T_{a}$ in such a way that the path from $node\brac{x}$ to $node\brac{y}$ goes
through the leftmost children of \emph{P}-nodes and the leftmost sections of \emph{Q}-nodes.
Then, we compute the string representation $\mathcal{S}$ of the modified tree,
and move the first occurrence of $y$ to the right and the second occurrence of $y$ to the left until both meet, as in the Lemma \ref{lemma:RemovingEdgePLeaf}.
After this operation is done, $y$ is the leftmost vertex of the tree $T_{a}$ - its right endpoint appears first in the string representation of $T_a$.
In order to restore all the removed edges except $\brac{x,y}$, we move the first occurrence of $x$ in $\mathcal{S}$ to the left,
until its predecessor is the second occurrence of $y$.
Clearly, this procedure restores all the removed edges except $\brac{x,y}$.

\begin{figure}[h]
\begin{center}
\begin{scriptsize}
\begin{tikzpicture}

	\draw [-] (0,1.8) -- (4.9,1.8);
	\draw [-] (0,0.8) -- (4.9,0.8);

    \draw [-] (0.75,0.8) -- (0.75,1.8);
	\node at (1,2.1) {\small $S_{l\brac{x}-1}$};
    \draw [-] (1.5,0.8) -- (1.5,1.8);
	\node at (2,2.1) {\small $S_{l\brac{x}}$};
    \draw [-] (2.25,0.8) -- (2.25,1.8);
    \draw [-] (3.4,0.8) -- (3.4,1.8);
    \node at (3.75,2.1) {\small $S_{a}$};
    \draw [-] (4.15,0.8) -- (4.15,1.8);

	\node at (3.75, 1.7) {$x$};
	\draw [|-] (1.7,1.6) -- (4.7,1.6);

    \draw [-|] (0.2,1.6) -- (1.3,1.6);
    \draw [|-,dotted] (1,1.4) -- (4.5,1.4);
    \draw [-|,dotted] (0.2,1.2) -- (4,1.2);
    \draw [|-,dotted] (1.8,1) -- (4.5,1);

    \draw [-] (1.1,0.8) -- (1.1,0.6);
	\node at (1,0.35) {\small $T_{l\brac{x}-1}$};
    \draw [-] (1.9,0.8) -- (1.9,0.6);
	\node at (2,0.35) {\small $T_{l\brac{x}}$};
    \draw [-] (3.75,0.8) -- (3.75,0.6);
    \node at (3.75,0.35) {\small $T_{a}$};

	\draw [-] (6,1.8) -- (10.9,1.8);
	\draw [-] (6,0.8) -- (10.9,0.8);

    \draw [-] (6.75,0.8) -- (6.75,1.8);
	\node at (7,2.1) {\small $S_{l\brac{x}-1}$};
    \draw [-] (7.5,0.8) -- (7.5,1.8);
    \node at (7.9,2.1) {\small $S_{a}$};
    \draw [-] (8.25,0.8) -- (8.25,1.8);
	\node at (8.75,2.1) {\small $S_{l\brac{x}}$};
    \draw [-] (9,0.8) -- (9,1.8);

	\node at (9.8, 1.7) {$x$};
	\draw [|-] (8.5,1.6) -- (10.7,1.6);

    \draw [-|] (6.2,1.6) -- (7.3,1.6);
    \draw [|-,dotted] (7,1.4) -- (10.5,1.4);
    \draw [-|,dotted] (6.2,1.2) -- (10.5,1.2);
    \draw [|-,dotted] (7.7,1) -- (10.5,1);

    \draw [-] (7.1,0.8) -- (7.1,0.6);
	\node at (7,0.35) {\small $T_{l\brac{x}-1}$};
    \draw [-] (7.9,0.8) -- (7.9,0.6);
    \node at (7.9,0.35) {\small $T_a$};
    \draw [-] (8.65,0.8) -- (8.65,0.6);
    \node at (8.75,0.35) {\small $T_{l\brac{x}}$};

    \node at (5.5,1.3) {\small $\rightarrow$};

\end{tikzpicture}
\end{scriptsize}
\end{center}
\caption{Removing an edge $\brac{x,y}$ in the case where $\tbrac{x,y}$-tree-path is almost rotable, starts in some $x$-central section $S_a$, and $y$ has some neighbor in a subtree $T_a$. (Case 1)}
\label{fig:StringQPathEdge}
\end{figure}
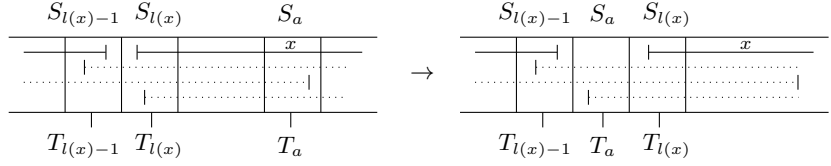

\end{proof}

The above two lemmas show the only cases where $\tbrac{x,y}$-tree-path is not rotable, but $\brac{x,y}$ is an interval edge.
Now, we are going to show that if $\tbrac{x,y}$-tree-path is not rotable, and the conditions of those lemmas are not satisfied, then $\brac{x,y}$ is not an interval edge.
At first, we prove that if $\tbrac{x,y}$-tree-path does not end in a \emph{P}-node that is a leaf, then $\brac{x,y}$ is not an interval edge.

\begin{lemma}
\label{lemma:IntervalEdge_QLeaf}
If an $\tbrac{x,y}$-tree-path ends in a \emph{Q}-node, then $\brac{x,y}$ is not an interval edge.
\end{lemma}

\begin{proof}
If $y$ belongs only to central sections ($y \notin S_1 \cup S_k$),
then Lemma \ref{lemma:QNodeProperties}f implies that there are vertices $z_1 \in \brac{S_{l\brac{y}} \cup V_{l\brac{y}}} \setminus S_{l\brac{y}+1}$
and $z_2 \in \brac{S_{r\brac{y}} \cup V_{r\brac{y}}} \setminus S_{r\brac{y}-1}$.
If $y$ belongs to $S_1$ (or $S_k$), then just take $z_1 \in V_1$ (or $z_2 \in V_k$).
Clearly, there is no edge between $z_1$ and $z_2$, and both $x$ and $y$ are over $z_1$ and $z_2$.
Thus, Observation \ref{obs:InducedC4} finishes the proof.
\end{proof}

\begin{lemma}
\label{lemma:IntervalEdge_PNotLeaf}
If an $\tbrac{x,y}$-tree-path ends in a \emph{P}-node that is not a leaf,
then $\brac{x,y}$ is not an interval edge.
\end{lemma}

\begin{proof}
The same argument as in Lemma \ref{lemma:RemovingEdgePInternal}.
\end{proof}

Now, we are going to prove that if an $\tbrac{x,y}$-tree-path goes through the central section,
then $\brac{x,y}$ is not an interval edge.
First, we define a \emph{nested path}.
Consider a \emph{Q}-node with $k$ sections $S_1,\ldots,S_k$.
For a clarity let $S_0 = S_{k+1} = \emptyset$, and say that a set of vertices
$\mathcal{P}_{i,j} = \brac{ S_{i} \cup \ldots \cup S_{j} } \setminus \brac{ S_{i-1} \cup S_{j+1}}$
is a \emph{nested path}, if for every $i \leq a < j$ there is at least one vertex $v_a \in \mathcal{P}_{i,j}$ that belongs to $S_{a} \cap S_{a+1}$.
We call it a nested path, because vertices $v_{i},v_{i+1}\ldots,v_{j-1}$ in that order form (possibly not simple) path in the graph $G$,
and all of them are contained in the sections $S_i,\ldots,S_j$.
In the following lemma, we show that if $\tbrac{x,y}$-tree-path goes through a central section of some \emph{Q}-node,
then we can find an asteroidal triple in the graph $G - \brac{x,y}$.
Nested paths help us to find this triple.

\begin{lemma}
\label{lemma:IntervalEdge_PathCentralSection}
If $\tbrac{x,y}$-tree-path goes through the central section, then $\brac{x,y}$ is not an interval edge.
\end{lemma}

\begin{proof}
Assume that the $\tbrac{x,y}$-tree-path goes through some central section $S_i$ of a \emph{Q}-node $Q$,
and let $S_1,\ldots,S_k$ be the sections of $Q$.
Subtrees of the first and last section are nonempty, so let $z_1 \in V_1$ and $z_2 \in V_k$.
We show that vertices $y,z_1$ and $z_2$ form an asteroidal triple in the graph $G - \brac{x,y}$.

There is no edge $\brac{x,y}$, so a path $z_1 - x - z_2$ avoids the neighborhood of $y$.
To prove that there is a path between $z_1$ and $y$ that avoids the neighborhood of $z_2$,
we show that there is a nested path $\mathcal{P}_{1,k-1}$,
and so the shortest path of form $z_1 - \mathcal{P}_{1,k-1} - y$ fulfill our requirements.
Let $\mathcal{P}_{1,j}$ be a nested path with maximum $j < k$.
Clearly, such a path exists.
Otherwise, either $S_1$ is empty, or all vertices that belong to $S_1$ belong to all sections.
In both cases, properties listed in Lemma \ref{lemma:QNodeProperties} are violated.
Moreover, if $j$ is less than $k-1$,
then either $S_{j} \cap S_{j+1} = \emptyset$, or $S_{j} \cap S_{j+1} \subset S_{k}$.
Again, this contradicts Lemma \ref{lemma:QNodeProperties}, and we are done.
A path from $z_2$ to $y$ that avoids the neighborhood of $z_1$ is constructed in a similar way.
\end{proof}

\begin{figure}[h]
\begin{center}
\begin{scriptsize}
\begin{tikzpicture}

    \draw [-] (0,1) -- (2.5,1);
    \draw [-] (0,1.5) -- (2.5,1.5);

    \draw [-] (0,1) -- (0,1.5);
    \draw [-] (0.5,1) -- (0.5,1.5);
    \draw [-] (1,1) -- (1,1.5);
    \draw [-] (1.5,1) -- (1.5,1.5);
    \draw [-] (2,1) -- (2,1.5);
    \draw [-] (2.5,1) -- (2.5,1.5);

    \node at (0.25,1.25) {\small $S_{1}$};
    \node at (0.75,1.25) {\ldots};
    \node at (1.25,1.25) {\small $S_{i}$};
    \node at (1.75,1.25) {\ldots};
    \node at (2.25,1.25) {\small $S_{k}$};

    \node at (0.25,0.5) {\small $T_{1}$};
    \node at (1.25,0) {\small $y$};
    \node at (2.3,0.5) {\small $T_{k}$};

    \node at (1.7,2.5) {\small $x$};

    \draw [-] (0.25,1) -- (0.25,0.7);
    \draw [-,dotted] (1.25,1) -- (1.25,0.3);
    \draw [-,dotted] (1.25,1.5) -- (1.65,2.3);
    \draw [-] (2.25,1) -- (2.25,0.7);

\end{tikzpicture}
\end{scriptsize}
\end{center}
\caption{An $\tbrac{x,y}$-tree-path which goes through a central section $S_i$ of some \emph{Q}-node.}
\label{fig:PathThroughCentralSection}
\end{figure}
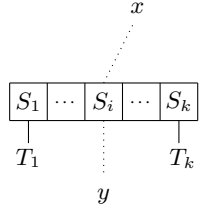

A consequence of the previous three lemmas is the following corollary.

\begin{corollary}
If an $\tbrac{x,y}$-tree-path is not almost rotable, then $\brac{x,y}$ is not an interval edge.
\end{corollary}

We are almost done with the classification.
In Lemma \ref{lemma:RemoveEdgePPath} we proved that if an $\tbrac{x,y}$-tree-path is rotable,
then $\brac{x,y}$ is always an interval edge.
On the other hand, in Lemmas \ref{lemma:IntervalEdge_QLeaf}, \ref{lemma:IntervalEdge_PNotLeaf} and \ref{lemma:IntervalEdge_PathCentralSection}
we considered $\tbrac{x,y}$-tree-paths that are not almost rotable,
and showed that for such paths $\brac{x,y}$ is not an interval edge.
Hence, the only remaining cases are the $\tbrac{x,y}$-tree-paths that are almost rotable, but not rotable.
In Lemmas \ref{lemma:RemoveEdgeInnerSectionNoNeighbour} and \ref{lemma:RemoveEdgeInnerSectionNeighbour}
we investigated such paths, and proved that under some additional conditions $\brac{x,y}$ is an interval edge.
Now we show that if an $\tbrac{x,y}$-tree-path is almost rotable, but is not rotable,
and conditions of Lemmas \ref{lemma:RemoveEdgeInnerSectionNoNeighbour} and \ref{lemma:RemoveEdgeInnerSectionNeighbour} are not satisfied, then $\brac{x,y}$ is not an interval edge.
This case is the hardest one, and before we prove it, we need to prove two auxiliary lemmas.

\begin{lemma}
\label{lemma:InnerQNode}
Let $S_1,\ldots,S_k$ be the sections of some \emph{Q}-node in an \emph{MPQ}-tree $\mathcal{T}$,
and for the simplicity assume that $S_0 = S_{k+1} = \emptyset$.
For every pair of indices $\brac{a,b} \neq \brac{1,k}$, $a < b$
there is a vertex $v \in \brac{ \brac{S_{a-1} \cap S_{a}} \setminus S_{b} } \cup \brac{ \brac{S_{b} \cap S_{b+1}} \setminus S_{a} }$.
\end{lemma}

\begin{proof}
By contradiction, assume that for some pair of indices $\brac{a,b}$ there is no such vertex,
so each vertex $u \in S_{a} \cup \ldots \cup S_{b}$
is either contained in $S_a,\ldots,S_b$, or belongs to $S_a \cap \ldots \cap S_b$.
Without loss of generality assume that $b < k$.
Each vertex belongs to at least two sections.
Hence, no vertex has its right endpoint in $S_{a}$ or left endpoint in $S_{b}$,
and Lemma \ref{lemma:QNodeProperties} implies that $V_a$, $V_b$ and $V_k$ are not empty.
Let $v_a \in V_a$, $v_b \in V_b$, $v_k \in V_k$,
and consider any maximal cliques $C_a,C_b$ and $C_k$ such that $v_a \in C_a$, $v_b \in C_b$ and $v_k \in C_k$.
Note that, $\mathcal{T}$ encodes only those orderings of maximal cliques in which either $C_a < C_b < C_k$ or $C_k < C_b < C_a$.
Now, consider a modified tree $\mathcal{T}'$ in which we reverse the order of sections $S_a,\ldots,S_b$.
Clearly, because of our assumptions, $\mathcal{T}'$ represents the same interval graph as $\mathcal{T}$,
but $\mathcal{T}'$ encodes an ordering $C_b < C_a < C_k$, which is not encoded by $\mathcal{T}$.
Thus, $\mathcal{T}$ is not a valid \emph{MPQ}-tree, and we are done.
\end{proof}

\begin{lemma}
\label{lemma:InnerInterval}
If an $\tbrac{x,y}$-tree-path starts in a central section $S_a$,
and there is a vertex $q \in S_a \setminus \brac{S_{l\brac{x}} \cup S_{r\brac{x}}}$,
then $\brac{x,y}$ is not an interval edge.
\end{lemma}

\begin{proof}
By contradiction, assume that $\brac{x,y}$ is an interval edge,
Let $P_{l_1,r_1}$ for $r_1 < a$ be a nested path that intersects $q$ and have the smallest $l_1$.
Analogously, let $P_{l_2,r_2}$ for $a < l_2$ be a nested path that intersects $q$ and have the biggest $r_2$.
Notice that these paths may not exist.
For instance, if $q$ has its right endpoint in $S_a$, then $P_{l_2,r_2}$ does not exist.

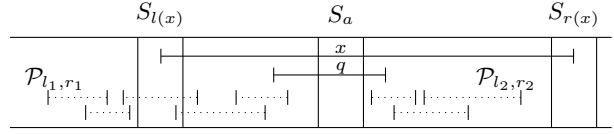
\begin{figure}[h]
\begin{center}
\begin{scriptsize}
\begin{tikzpicture}


	\draw [-] (0,2.2) -- (8,2.2);
	\draw [-] (0,1) -- (8,1);

	\draw [-] (4.1,1) -- (4.1,2.2);
	\draw [-] (4.7,1) -- (4.7,2.2);
    \node at (4.4,2.5) {\small $S_a$};

	\draw [-] (1.7,1) -- (1.7,2.2);
	\draw [-] (2.3,1) -- (2.3,2.2);
	\node at (2,2.5) {\small $S_{l\brac{x}}$};

	\draw [-] (7.2,1) -- (7.2,2.2);
	\draw [-] (7.8,1) -- (7.8,2.2);
	\node at (7.5,2.5) {\small $S_{r\brac{x}}$};

	\draw [|-|] (2,1.95) -- (7.5,1.95);
	\node at (4.4, 2.05) {$x$};

	\draw [|-|] (3.5,1.7) -- (5,1.7);
	\node at (4.4, 1.8) {$q$};

    \node at (6.6,1.65) {\small $\mathcal{P}_{l_2,r_2}$};
	\draw [|-|,dotted] (4.8,1.4) -- (5.4,1.4);
	\draw [|-|,dotted] (5.1,1.2) -- (6.1,1.2);
	\draw [|-|,dotted] (5.5,1.4) -- (6.8,1.4);

    \node at (0.6,1.65) {\small $\mathcal{P}_{l_1,r_1}$};
	\draw [|-|,dotted] (3.7,1.4) -- (3,1.4);
	\draw [|-|,dotted] (3.4,1.2) -- (2.2,1.2);
	\draw [|-|,dotted] (2.5,1.4) -- (1.5,1.4);
	\draw [|-|,dotted] (1.6,1.2) -- (1,1.2);
	\draw [|-|,dotted] (1.3,1.4) -- (0.5,1.4);

\end{tikzpicture}
\end{scriptsize}
\end{center}
\caption{Paths $\mathcal{P}_{l_1,r_1}$ and $\mathcal{P}_{l_2,r_2}$.}
\label{fig:InnerPaths}
\end{figure}

At first, we consider the case where both paths do not exist, and prove that there is an asteroidal triple in the graph $G - \brac{x,y}$.
By Lemma \ref{lemma:InnerQNode}, without loss of generality, there is a vertex
$v \in \brac{S_{r\brac{q}} \cap S_{r\brac{q}+1}} \setminus S_{l\brac{q}}$.
Both paths do not exist, so $v$ has to belong to $S_a$.
Moreover, Lemma \ref{lemma:QNodeProperties}f gives two vertices $v_L \in \brac{S_{l\brac{q}} \cup V_{l\brac{q}}} \setminus S_{l\brac{q}+1}$,
and $v_R \in \brac{S_{r\brac{q}+1} \cup V_{r\brac{q}+1}} \setminus S_{r\brac{q}}$.
Thus, the paths:
$v_L - x - v_R$,
$v_L - q - y$, and
$v_R - v - y$
certify the asteroidal triple $\set{v_L,v_R,y}$, and we are done in this case.

Now, assume that both paths exist, and put $L = \max\set{l_1,l\brac{x}}$ and $R = \min\set{r_2,r\brac{x}}$.
Lemma \ref{lemma:QNodeProperties}f implies that there is a vertex $v_L$ which has its right endpoint in $S_L$ or $v_L \in V_L$.
Analogously, define a vertex $v_R$ for the section $S_R$.
Clearly, both $v_L$ and $v_R$ are neighbors of $x$, but they do not belong to the neighborhood of $q$.
Hence, $\set{v_L,v_R,y}$ is an asteroidal triple in the graph $G - \brac{x,y}$.
To see this, consider the paths:
$v_L - x - v_R$,
$v_L - P_{l_1,r_1} - q - y$, and
$v_R - P_{l_2,r_2} - q - y$.
Thus, if both paths exist, then $\brac{x,y}$ is not an interval edge.

Without loss of generality, assume that $\mathcal{P}_{l_1,r_1}$ exists, but $\mathcal{P}_{l_2,r_2}$ does not.
Lemma \ref{lemma:InnerQNode} implies that there is a vertex $v$ such that
$v \in \brac{S_{l_1-1} \cap S_{l_1}} \setminus S_{r\brac{q}}$, or
$v \in \brac{S_{r\brac{q}} \cap S_{r\brac{q}+1}} \setminus S_{l_1}$.
Note that, a this point $v$ and $x$ might be the same vertex.
We assumed that right path does not exist, so $S_{r\brac{q}} \cap S_{r\brac{q}+1} \subset S_{a}$.
Moreover, $l_1$ is the smallest index such that there is a nested path $P_{l_1,r_1}$ for $r_1<a$ that intersects $q$.
Hence, $S_{l_1-1} \cap S_{l_1} \subset S_{a}$, and so in both cases $v$ belongs to $S_a$.
Lemma \ref{lemma:QNodeProperties}f gives three vertices:
$v_L \in \brac{S_{l_1} \cup V_{l_1}} \setminus S_{l_1-1}$,
$v_R \in \brac{S_{r\brac{q}} \cup V_{r\brac{q}}} \setminus S_{r\brac{q} - 1}$, and
$v_{R+1} \in \brac{S_{r\brac{q}+1} \cup V_{r\brac{q}+1}} \setminus S_{r\brac{q}}$.
Note that, both $v_R$ and $v_{R+1}$ are neighbors of $x$, and both $v_L$ and $v_{R+1}$ are not neighbors of $q$.

If $l_1 < l\brac{x}$,
then the paths:
$v_L - P_{l_1,r_1} - x - v_{R+1}$,
$v_L - P_{l_1,r_1} - q - y $, and
$v_{R+1} - x - q - y$ certify the asteroidal triple $\set{v_L,v_{R+1},y}$.
Thus, we conclude that $l\brac{x} \leq l_1 < l\brac{q}$,
which means that the whole path $P_{l_1,r_1}$ is contained in $x$.
Moreover, $v$ and $x$ are different vertices, and $v_L$ is a neighbor of $x$.
If $v \in \brac{S_{l_1-1} \cap S_{l_1}} \setminus S_{r\brac{q}}$,
then there is an asteroidal triple $\set{v_L,v_R,y}$ certified by paths: $v_L - x - v_R$, $v_L - v - y$ and $v_R - q - y$.
On the other hand, if $v \in \brac{S_{r\brac{q}} \cap S_{r\brac{q}+1}} \setminus S_{l_1}$,
then $\set{v_L,v_{R+1},y}$ is an asteroidal triple certified by paths:
$v_L - x - v_{R+1}$,
$v_L - P_{l_1,r_1} - q - y$, and
$v_{R+1} - v - y$.
Thus, in this case the edge $\brac{x,y}$ is also not an interval edge, and we are done.
\end{proof}


Finally, we are ready to prove the last two negative results on interval edges.

\begin{lemma}
\label{lemma:RemoveEdgeImpossibleYHasEdge}
If an $\tbrac{x,y}$-tree-path
starts in a central section $S_a$,
$y$ has a neighbor in subtree $T_a$,
and:
\begin{enumerate}
    \item $i > 1 \implies \brac{ \exists_{z_i \neq x}: z_i \in S_{a} \setminus S_{i} \text{ or } \exists_{\bar{z_i}}: \bar{z_i} \in \brac{S_{i-1} \cap S_{i}} \setminus S_{a} }$, and
    \item $j < k \implies \brac{ \exists_{z_j \neq x}: z_j \in S_{a} \setminus S_{j} \text{ or } \exists_{\bar{z_j}}: \bar{z_j} \in \brac{S_{j} \cap S_{j+1}} \setminus S_{a} }$, and
    \item $i = 1 \implies \exists_{z_1 \neq x}: z_1 \in S_{a} \setminus S_1$, and
    \item $j = k \implies \exists_{z_k \neq x}: z_k \in S_{a} \setminus S_k$,
\end{enumerate}
where $i = l\brac{x}$ and $j = r\brac{x}$,
then $\brac{x,y}$ is not an interval edge.
\end{lemma}

\begin{proof}
By contradiction, assume that $\brac{x,y}$ is an interval edge, and let $z$ be the neighbor of $y$ in $T_a$.
If both vertices $z_i$ and $z_j$ exist, and $z_i = z_j$,
then $z_i \in S_{a} \setminus \brac{S_{i} \cup S_{j}}$, and Lemma \ref{lemma:InnerInterval} leads to a contradiction.
Hence, if both of them exist, then $z_i \neq z_j$.
But, in this case we can find an asteroidal triple $\set{v_L,v_R,y}$, where $v_L \in \brac{S_{i} \cup V_{i}} \setminus S_{i+1}$ and $v_R \in \brac{S_{j} \cup V_{j}} \setminus S_{j-1}$,
certified by paths:
$v_L - x - v_R$,
$v_L - z_j - y$, and
$v_R - z_i - y$.

Thus, without loss of generality $z_j$ does not exist, and there is a vertex $\bar{z_j}$ which belongs to $\brac{S_{j} \cap S_{j+1}} \setminus S_{a}$.
If $\bar{z_i}$ also exists,
then consider vertices
$v_L \in \brac{S_{i-1} \cup V_{i-1}} \setminus S_{i}$,
and $v_R \in \brac{S_{j+1} \cup V_{j+1}} \setminus S_{j}$ that are given by Lemma \ref{lemma:QNodeProperties}f,
and notice that paths:
$v_L - \bar{z_i} - x - \bar{z_j} - v_R$,
$v_L - \bar{z_i} - x - z - y$, and
$v_R - \bar{z_j} - x - z - y$
certify the asteroidal triple $\set{v_L,v_R,y}$.

Hence, $\bar{z_i}$ does not exist, and the only remaining case is where $\bar{z_j}$ and $z_i$ exist.
Note that if $z_i$ does not belong to $S_j$, then $z_i \in S_{a} \setminus \brac{S_i \cup S_j}$ and Lemma \ref{lemma:InnerInterval} gives a contradiction.
Hence, $z_i \in S_{j}$ and there is an edge between $z_i$ and $\bar{z_j}$.
Lemma \ref{lemma:QNodeProperties}f gives vertices $v_L \in \brac{S_{i} \cup V_{i}} \setminus S_{i+1}$ and $v_R \in \brac{S_{j+1} \cup V_{j+1}} \setminus S_{j}$,
and the asteroidal triple $\set{v_L,v_R,y}$ is certified by paths:
$v_L - x - \bar{z_j} - v_R$,
$v_L - x - z - y$, and
$v_R - \bar{z_j} - z_i - y$.
\end{proof}


\begin{lemma}
\label{lemma:RemoveEdgeImpossibleYHasNoEdge}
If an $\tbrac{x,y}$-tree-path
starts in a central section $S_a$,
$y$ does not have a neighbor in subtree $T_a$,
and:
\begin{enumerate}
    \item $\forall_{1 < b \leq l\brac{x}}: \exists_{z_b \neq x}: z_b \in S_{a} \setminus S_{b}$ or $\exists_{\bar{z_b}}: \bar{z_b} \in \brac{S_{b-1} \cap S_{b}} \setminus S_{a}$, and
    \item $\forall_{r\brac{x} \leq b < k}: \exists_{z_b \neq x}: z_b \in S_{a} \setminus S_{b}$ or $\exists_{\bar{z_b}}: \bar{z_b} \in \brac{S_{b} \cap S_{b+1}} \setminus S_{a}$, and
    \item $\exists_{z_1 \neq x}: z_1 \in S_{a} \setminus S_1$, and
    \item $\exists_{z_k \neq x}: z_k \in S_{a} \setminus S_k$,
\end{enumerate}
then $\brac{x,y}$ is not an interval edge.
\end{lemma}

\begin{proof}
By contradiction, assume that $\brac{x,y}$ is an interval edge.
Let $p_L$ be the longest sequence $\bar{z_{l\brac{x}}},\bar{z_{l\brac{x}-1}},\ldots,\bar{z_L}$,
and $p_R$ be the longest sequence $\bar{z_{r\brac{x}}},\bar{z_{r\brac{x}+1}},\ldots,\bar{z_R}$.
If both sequences are empty, then there are vertices $z_{l\brac{x}} \in S_{a} \setminus S_{l\brac{x}}$ and $z_{r\brac{x}} \in S_{a} \setminus S_{r\brac{x}}$.
But, the same argument as in Lemma \ref{lemma:RemoveEdgeImpossibleYHasEdge}, shows that there is an asteroidal triple in the graph $G - \brac{x,y}$.
Thus, without loss of generality $p_L$ is not empty.
If both sequences are not empty, then Lemma \ref{lemma:QNodeProperties}f gives vertices $v_L \in \brac{S_{L-1} \cup V_{L-1}} \setminus S_{L}$
and $v_R \in \brac{S_{R+1} \cup V_{R+1}} \setminus S_{R}$.
Moreover, both sequences are the longest possible, so there are vertices $z_{L-1} \in S_{a} \setminus S_{L-1}$ and $z_{R+1} \in S_{a} \setminus S_{R+1}$.
Thus, the paths
$v_L - p_L - x - p_R - v_R$,
$v_L - p_L - x - z_{R+1} - y$, and
$v_R - p_R - x - z_{L-1} - y$ certify the asteroidal triple $\set{v_L,v_R,y}$.
Hence, $p_R$ is empty, but $p_L$ is not.
In that case, let $v_R$ be a vertex that belongs to $S_{r\brac{x}} \cup V_{r\brac{x}}$, but does not belong to $S_{r\brac{x}-1}$.
Sequence $p_R$ is empty, so there is a vertex $z_{l\brac{x}}$ that belongs to $S_{a}$, but does not belong to $S_{r\brac{x}}$.
Yet, again we can find an asteroidal triple $\set{v_L,v_R,y}$ certified by paths:
$v_L - p_L - x - v_R$,
$v_L - p_L - x - z_{r\brac{x}} - y$, and
$v_R - x - z_{L-1} - y$.
\end{proof}

Finally, we are done with the classification of interval edges.
For each edge $\brac{x,y}$ which is an interval edge,
we provided a linear time algorithm that produces a string representation for the graph $G - \brac{x,y}$,
see Lemmas \ref{lemma:RemovingEdgePLeaf}, \ref{lemma:RemovingEdgeQNode}, \ref{lemma:RemoveEdgePPath}, \ref{lemma:RemoveEdgeInnerSectionNoNeighbour}, and \ref{lemma:RemoveEdgeInnerSectionNeighbour}.
Moreover, for every edge $\brac{x,y}$ that is not an interval edge
we presented structures certifying that $G - \brac{x,y}$ is not an interval graph,
see Lemmas \ref{lemma:RemovingEdgePInternal}, \ref{lemma:RemovingEdgeQFail}, \ref{lemma:IntervalEdge_QLeaf}, \ref{lemma:IntervalEdge_PNotLeaf}, \ref{lemma:IntervalEdge_PathCentralSection}, \ref{lemma:RemoveEdgeImpossibleYHasEdge} and \ref{lemma:RemoveEdgeImpossibleYHasNoEdge}.
The consequence of all those lemmas is the following theorem.

\begin{theorem}
\label{thm:AlgorithmStringRepresentation}
There is a linear time algorithm, that for every \emph{MPQ}-tree $\mathcal{T}$ representing graph $G$, and every interval edge $\brac{x,y}$,
produces a string representation $\mathcal{S}'$ for the graph $G - \brac{x,y}$.
\end{theorem}

\section{Listing interval edges}
\label{sec:Listing}

In this section we present an efficient algorithm that for a given \emph{MPQ}-tree $\mathcal{T}$ lists all interval edges of the graph represented by $\mathcal{T}$.
Let $S_{1},\ldots,S_{k}$ be the sections of a \emph{Q}-node $Q$.
For $b \in \set{2,\ldots,k}$, denote $f_{r}\brac{b}$ the maximal index of a section
such that $S_{b-1} \cap S_{b} \subset S_{f_{r}\brac{b}}$.
Analogously, for $b \in \set{1,\ldots,k-1}$, denote $f_{l}\brac{b}$ the minimal index of a section
such that $S_{b} \cap S_{b+1} \subset S_{f_{l}\brac{b}}$.
For every vertex $x$ which belongs to $Q$, and every $l\brac{x} \leq a \leq r\brac{x}$
let:

\begin{minipage}{0.45\textwidth}
\[
    L^{*}\brac{x,a} =
    \begin{dcases}
	\omit\hfil 1 \hfil, & S_{a} = \set{x} \\
	\max_{ v \in S_{a} \setminus \set{x} } l\brac{v}, & S_{a} \setminus \set{x} \neq \emptyset \\
    \end{dcases}
\]
\end{minipage}%
\hfill
\begin{minipage}{0.45\textwidth}
\[
    R^{*}\brac{x,a} =
    \begin{dcases}
	\omit\hfil k \hfil, & S_{a} = \set{x} \\
	\min_{ v \in S_{a} \setminus \set{x} } r\brac{v}, & S_{a} \setminus \set{x} \neq \emptyset \\
    \end{dcases}
\]
\end{minipage}%

In other words, if $S_{a} \neq \set{x}$,
then $L^{*}\brac{x,a} = i$ if $S_i$ is the rightmost section such that
there is a vertex $v \neq x$ which belongs to $S_{a}$ and $v$ has its left endpoint in $S_i$.
Now, we are ready to express the inclusion conditions on \emph{Q}-node sections in terms of the functions
$f_l$, $f_r$, $L^{*}$, and $R^{*}$.
\begin{observation}
\label{obs:Reformulation}
In respect to the above definitions, we have the following equivalences:
\begin{enumerate}
\item $\forall_{1 < b \leq a} : S_{b-1} \cap S_{b} \subset S_{a} \iff f_{r}\brac{b} \geq a$.
\item $\forall_{a \leq b < k} : S_{b} \cap S_{b+1} \subset S_{a} \iff f_{l}\brac{b} \leq a$.
\item $\forall_{l\brac{x} \leq a \leq r\brac{x}} : S_{a} \setminus \set{x} \subset S_{b} \iff b \in \sbrac{L^{*}\brac{x,a},R^{*}\brac{x,a}}$.
\end{enumerate}
\end{observation}

Using those equivalences, we are able to reformulate the conditions of Lemmas
\ref{lemma:RemoveEdgeInnerSectionNoNeighbour} and \ref{lemma:RemoveEdgeInnerSectionNeighbour}.

\begin{observation}
\label{obs:NoNeighbour}
For an $\tbrac{x,y}$-tree-path that starts in a central section $S_a$,
the conditions (1), (2), (3), and (4) of Lemma \ref{lemma:RemoveEdgeInnerSectionNoNeighbour} are equivalent to the following:
\begin{enumerate}
\item $L^{*}\brac{x,a} \leq l\brac{x}$ and $\min\set{l\brac{x},R^{*}\brac{x,a}} > 1$ and $f_{r}\brac{\min\set{l\brac{x},R^{*}\brac{x,a}}} \geq a$, or
\item $R^{*}\brac{x,a} \geq r\brac{x}$ and $\max\set{r\brac{x},L^{*}\brac{x,a}} < k$ and $f_{l}\brac{\max\set{r\brac{x},L^{*}\brac{x,a}}} \leq a$, or
\item $L^{*}\brac{x,a} = 1$, or
\item $R^{*}\brac{x,a} = k$.
\end{enumerate}
\end{observation}

\begin{proof}
Note that, conditions (3) and (4) translate directly, and conditions (1) and (2) are symmetrical.
Hence, we only show the equivalence of the first cases.
At first, we show that if there is $1 < b \leq l\brac{x}$ such that $S_{a} \setminus \set{x} \subset S_{b}$
and $S_{b-1} \cap S_{b} \subset S_{a}$,
then $L^{*}\brac{x,a} \leq l\brac{x}$ and $f_{r}\brac{\min\set{l\brac{x},R^{*}\brac{x,a}}} \geq a$.
From Observation \ref{obs:Reformulation} we have
$b \in \sbrac{L^{*}\brac{x,a},R^{*}\brac{x,a}}$ and $f_{r}\brac{b} \geq a$.
Thus, $L^{*}\brac{x,a} \leq l\brac{x}$, and $1 < b \leq \min \set{l\brac{x},R^{*}\brac{x,a}}$.
Note that $f_{r}$ is a non-decreasing function, so $f_{r}\brac{b} \geq a$,
implies $f_{r}\brac{\min\set{l\brac{x},R^{*}\brac{x,a}}} \geq a$, and we are done.

Now, we take $b = \min\set{l\brac{x},R^{*}\brac{x,a}}$,
and show that if $b > 1$, $f_{r}\brac{b} \geq a$, and $L^{*}\brac{x,a} \leq l\brac{x}$,
then $S_{b-1} \cap S_{b} \subset S_{a}$ and $S_{a} \setminus \set{x} \subset S_{b}$.
Clearly, $1 < b \leq l\brac{x}$ and $f_r\brac{b} \geq a$, so $S_{b-1} \cap S_{b} \subset S_{a}$.
Hence, the only thing we need to show is $S_{a} \setminus \set{x} \subset S_{b}$.
According to Observation \ref{obs:Reformulation} it is equivalent to
$b \in \sbrac{L^{*}\brac{x,a},R^{*}\brac{x,a}}$.
It is easy to see that for every $x$ and $a$ we have $L^{*}\brac{x,a} \leq R^{*}\brac{x,a}$.
Thus, the interval $\sbrac{L^{*}\brac{x,a},R^{*}\brac{x,a}}$ is never empty.
If $b = R^{*}\brac{x,a}$, then we are done, so assume that $b = l\brac{x}$.
But, in this case $L^{*}\brac{x,a} \leq l\brac{x} = b \leq R^{*}\brac{x,a}$, and we are done too.
\end{proof}

A very similar proof applies to the following observation, so we leave it to the reader.

\begin{observation}
\label{obs:Neighbour}
For an $\tbrac{x,y}$-tree-path that starts in a central section $S_a$,
the conditions (1), (2), (3), and (4) of Lemma \ref{lemma:RemoveEdgeInnerSectionNeighbour} are equivalent to the following:
\begin{enumerate}
\item $l\brac{x} > 1$ and $l\brac{x} \in \sbrac{L^{*}\brac{x,a},R^{*}\brac{x,a}}$ and $f_{r}\brac{l\brac{x}} \geq a$, or
\item $r\brac{x} < k$ and $r\brac{x} \in \sbrac{L^{*}\brac{x,a},R^{*}\brac{x,a}}$ and $f_{l}\brac{r\brac{x}} \leq a$, or
\item $l\brac{x} = 1$ and $L^{*}\brac{x,a} = 1$, or
\item $r\brac{x} = k$ and $R^{*}\brac{x,a} = k$.
\end{enumerate}
\end{observation}


Observations \ref{obs:NoNeighbour} and \ref{obs:Neighbour} provide fast and easy tests under
the assumption that all the functions $l$, $r$, $f_l$, $f_r$, $L^{*}$ and $R^{*}$ are already computed.
To represent functions $l$, $r$, $f_l$ and $f_r$ in computer's memory we use $\Oh{n}$-element arrays of integers.
Functions $L^{*}$ and $R^{*}$ in their explicit forms may require $\Oh{n^2}$ space.
Hence, for performance purposes, we do not represent those functions as two-dimensional arrays.
Instead, we observe that the function $L^{*}$ can be defined using some one-dimensional functions $L_{1}$ and $L_{2}$.
We set $L_{1}\brac{a} = \max\set{l\brac{v} : v \in S_{a}}$, and denote $v_{1}\brac{a}$ a vertex for which $l\brac{v_{1}\brac{a}} = L_{1}\brac{a}$.
We also define $L_{2}\brac{a} = \max\set{l\brac{v} : v \in S_{a} \setminus \set{v_{1}\brac{a}}}$,
or $L_{2}\brac{a} = 1$ if $S_{a} = \set{v_{1}\brac{a}}$.
Using those two functions we are able to compute the function $L^{*}$ in the following way:
\[
    L^{*}\brac{x,a} =
    \begin{dcases}
        L_{1}\brac{a}, & L_{1}\brac{a} \neq l\brac{x} \\
        L_{2}\brac{a}, & \text{otherwise} \\
    \end{dcases}
\]
Analogously, we can represent the function $R^{*}$ using two functions $R_{1}$ and $R_{2}$.
Now we are ready to show the interval edges enumeration algorithm.

\begin{lemma}
\label{lemma:PartialFunctions}
Functions $l$, $r$, $f_l$ and $f_r$ for all \emph{Q}-nodes of the tree $\mathcal{T}$ can be computed in $\Oh{n \log n}$ time.
\end{lemma}

\begin{proof}
We compute these functions for each \emph{Q}-node separately.
Assume that we are computing the function $f_r$ for a \emph{Q}-node $Q_{d}$.
Let $i$ be the minimum index of the section containing a right endpoint of some vertex from $S_{b-1} \cap S_{b}$.
Clearly, $S_{b-1} \cap S_{b} \subset S_{i}$ and $\neg\brac{S_{b-1} \cap S_{b} \subset S_{i+1}}$.
Hence, $f_r\brac{b} = i$, and in order to compute the function $f_r$ for all $b$,
it is enough to scan the sections of $Q_{d}$ from left to right maintaining a heap of right endpoints.
When algorithm enters the section $S_{b+1}$ it adds to the heap all right endpoints of vertices that have its left endpoint in $S_{b}$,
assigns $f_{r}\brac{b}$ to be the minimum value stored in the heap,
and removes all the endpoints of vertices that have its right endpoint in $S_{b}$.
Thus, the computation of the function $f_r$ for $Q_{d}$ takes $\Oh{n_d \log n_d}$ time,
where $n_d$ denotes the number of vertices in $Q_{d}$.
As $\sum n_d \leq n$, the computation of all functions $f_r$ takes no more than $\Oh{n \log n}$ time.
A similar argument applies to the functions $f_l$.
\end{proof}

\begin{lemma}
\label{lemma:PartialLR}
Functions $L_{1}$, $L_{2}$, $R_{1}$ and $R_{2}$ for all \emph{Q}-nodes of the tree $\mathcal{T}$ can be computed in $\Oh{n \log n}$ time.
\end{lemma}

\begin{proof}
The proof is very similar to the proof of Lemma \ref{lemma:PartialFunctions}.
Yet again, we use a heap of right or left endpoints,
but now we are not only interested in the minimal element but also in the second one.
\end{proof}


\begin{theorem}
\label{thm:AlgorithmListEdges}
There is an algorithm that for a given \emph{MPQ}-tree $\mathcal{T}$ lists all its interval edges
in \newline$\Oh{\max \set{n+m, n \log n}}$ time.
\end{theorem}

\begin{proof}
The algorithm is pretty straightforward.
At first, we compute all the functions $l$, $r$, $f_l$, $f_r$, $L_1$, $L_2$, $R_1$ and $R_2$.
%
Then, we inspect all \emph{P}-nodes and all sections of \emph{Q}-nodes listing all edges $\brac{x,y}$
that satisfy the conditions of Lemmas \ref{lemma:RemovingEdgePLeaf} or \ref{lemma:RemovingEdgeQNode}.
For a \emph{P}-node that is a leaf, we list all pairs $\brac{v_a,v_b}$ for $a \neq b$,
while for the section $S_i$ with subtree that is either empty or is a \emph{P}-node with no children,
we list the edges of the form $\brac{v_L,v_R}$,
where $v_L$ is a vertex that has its right endpoint in $S_i$,
and $v_R$ is a vertex that has its left endpoint in $S_i$.
Note that, an \emph{MPQ}-tree has no more than $\Oh{n}$ nodes and sections.
Moreover, we have a constant time access to all vertices $v_L$ and $v_R$.
Thus, this phase works in $\Oh{n+m}$.

Now, we want to find all interval edges $\brac{x,y}$
such that $x$ and $y$ belong to different nodes in $\mathcal{T}$, and $x$ is over $y$.
Lemmas \ref{lemma:IntervalEdge_QLeaf} and \ref{lemma:IntervalEdge_PNotLeaf}
imply that it is enough to consider only those pair of vertices $\brac{x,y}$,
where $y$ belongs to a leaf of $\mathcal{T}$.
Hence, for each leaf $L$ of a tree $\mathcal{T}$,
we traverse a unique path between $L$ and the root of $\mathcal{T}$,
starting from $L$, and listing edges of the form $\brac{x,y}$,
where $x$ belongs to the currently visited node, and $y$ belongs to $L$.
We do not list all such edges, but for each candidate we need to decide whether $\brac{x,y}$ is an interval edge or not.
In order to do it efficiently, we keep two boolean variables:
1. does $y$ have a neighbor in the visited subtree, and
2. does the path go through some central section.
Using those two variables and precomputed functions,
we can decide whether $\brac{x,y}$ is an interval edge in a constant time thanks to
Lemma \ref{lemma:RemoveEdgePPath}, and Observations \ref{obs:NoNeighbour} and \ref{obs:Neighbour}.
Thus, the time spent by the algorithm in this phase is bounded from above by the sum of lengths of all paths we have visited plus the number of tested edges.
Lemma \ref{lemma:QNodeProperties}g, implies that on those paths there are no two consecutive empty \emph{P}-nodes,
so we can bound the sum of lengths of those paths by $\Oh{m}$.
Thus, the time spent by the algorithm in this phase is $\Oh{n+m}$,
and the whole algorithm works in $\Oh{\max\set{n+m,n \log n}}$ time.
\end{proof}

\section{Parent-Child Relationship}
\label{sec:ParentChild}
In this section we present the graph enumeration algorithm.
We define a parent-child relationship in the same way as authors in \cite{Yamazaki19} did using the following lemma.

\begin{lemma}[\cite{Kiyomi06}]
\label{lemma:ParentEdge}
If $G=\brac{V,E}$ is an interval graph which is not a clique,
then there is at least one edge $e \in E$ such that $G + e$ is also an interval graph.
\end{lemma}

\begin{theorem}[\cite{Yamazaki19} Thm. 5]
\label{thm:YamazakiParent}
Let $G = \brac{V,E}$ be any interval graph. Then its parent can be computed in $\Oh{n+m}$ time.
\end{theorem}

Yamazaki et al. used Lemma \ref{lemma:ParentEdge} and defined the parent of $G$ to be a graph $G + e$
such that $G + e$ is an interval graph and $e$ is lexicographically the smallest possible.
They proved that for every interval graph $G$ its parent can be computed in $\Oh{n+m}$ time - see Theorem \ref{thm:YamazakiParent}.
Thanks to the fact that we work with \emph{MPQ}-trees and string representations,
we are able to get rid of the $\Oh{m}$ factor using the algorithm from Theorem \ref{thm:LinearMPQTree}.

\begin{theorem}
\label{thm:AlgorithmParent}
There is a linear time algorithm that for every canonical \emph{MPQ}-tree $\mathcal{T}$ representing an interval graph $G$ that is not a clique,
produces a string representation $\mathcal{S}'$ of a graph $G + e$, where $e$ is the lexicographically smallest edge.
\end{theorem}

\begin{proof}
Let $\mathcal{S}$ be the canonical string for $\mathcal{T}$.
Let $12...j$ be the longest prefix of $\mathcal{S}$ such that $j...21$ is a suffix of $\mathcal{S}$.
Clearly, such prefix can be found in $\Oh{n}$ time.
Moreover, all vertices $1,\ldots,j$ have degree $n-1$ in the graph $G$.
Hence there is no pair $\brac{x,y} \notin E$ such that $x \in \set{1,\ldots,j}$.
We prove that there is a pair $\brac{j+1,y}$ for some $y \geq j+2$ such that $G + \brac{j+1,y}$ is an interval graph.
Let $y$ be the leftmost endpoint of an interval that is to the right of the right endpoint of $j+1$.
If such $y$ does not exist, then to the right of the right endpoint of $j+1$ there are only right endpoints.
Hence, $G$ is a clique and we reached a contradiction.
Thus, $y$ exists, $\brac{j+1,k}$ is an edge for all $k \leq y$, and there is no edge between $j+1$ and $y$.
So, $G + \brac{j+1,y}$ is the parent of $G$.
In order to produce a string representation for the graph $G + \brac{j+1,y}$,
we move the right endpoint of $j+1$ to the right until it passes the left endpoint of $y$.
Lemma \ref{lemma:SwappingLemma} guarantees that no edge except $\brac{j+1,y}$ was added.
\end{proof}

Finally, we take all the pieces together and present our main result.

\begin{theorem}
Let $\mathcal{T}$ be a canonical \emph{MPQ}-tree representing a non-empty ($m \geq 1$) interval graph $G$.
The set of canonical trees for children of $G$ in the family tree $\mathcal{F}_{n}$ can be computed in $\Oh{nm \log n}$ time.
\end{theorem}

\begin{proof}
At first, we list all interval edges for the tree $\mathcal{T}$ in $\Oh{\max\set{n+m, n \log n}}$ time, see Theorem \ref{thm:AlgorithmListEdges}.
Then, for each interval edge $e$, we produce a string representation $\mathcal{S}'$ of $G - e$ in linear time,
build \emph{MPQ}-tree $\mathcal{T}'$ for the interval graph represented by $\mathcal{S}'$ in $\Oh{n}$,
and finally compute the canonical form of $\mathcal{T}'$ in $\Oh{n \log n}$ time,
see Theorems \ref{thm:AlgorithmStringRepresentation}, \ref{thm:LinearMPQTree}, and \ref{thm:MPQTreeCanonicalForm}.
Hence, we compute all canonical strings for children candidates in $\Oh{nm \log n}$ time.
Theorem \ref{thm:TreeIsomorphism} implies that in order to check isomorphism, it is enough to remove duplicates from this set.
It can be easily done by storing all computed strings in a trie.
Finally, we need to filter out those graphs for which $G$ is not a parent.
For each candidate canonical string, we compute its parent string in $\Oh{n}$ time, see Theorem \ref{thm:AlgorithmParent},
build a canonical \emph{MPQ}-tree in $\Oh{n \log n}$ time and check whether its canonical string equals with a canonical string for the graph $G$ in $\Oh{n}$ time.
Thus, we filter out non-children in $\Oh{nm \log n}$ time, and the whole process takes $\Oh{nm \log n}$ time.
\end{proof}

\section{Performance}

In the previous section we presented theoretical analysis of our enumeration algorithm.
In this section we present two lemmas, that helped us to significantly speedup the execution of the algorithm.
Unfortunately, presented tricks do not improve the worst-case time delay,
and it is quite easy to show an \emph{MPQ}-tree that even with those tricks take $\Oh{nm \log n}$ time to process.
We leave it as an exercise.

\begin{lemma}
Let $v_1 \neq v_2$ be two vertices of the graph $G = \brac{V,E}$.
If $N\brac{v_1} \setminus \set{v_2} = N\brac{v_2} \setminus \set{v_1}$,
then for every $y \notin \set{v_1,v_2}$ graphs $G - \brac{v_1,y}$ and $G - \brac{v_2,y}$ are isomorphic.
\end{lemma}

\begin{proof}
One can easily check that the function $f: V \rightarrow V$ such that $f\brac{v_1} = v_2$, $f\brac{v_2} = v_1$ and $f$ identifies other vertices,
encodes an isomorphism between $G - \brac{v_1,y}$ and $G - \brac{v_2,y}$.
\end{proof}

A consequence of the above lemma is the fact that if a \emph{P}-node contains more than one vertex namely $v_1,\ldots,v_j$,
then graphs $G - \brac{v_1,y}$ and $G - \brac{v_a,y}$ for every $a > 1$ are isomorphic.
Thus, when listing potential candidates for the children of the graph $G$,
we may omit all edges of the form $\brac{v_a,y}$ for $a > 1$ and list only the edges of the form $\brac{v_1,y}$.
The same argument applies to a \emph{Q}-node and vertices that belong to the same sections.

\begin{lemma}
Let $v_1,\ldots,v_k$ be a subset of vertices of the graph $G = \brac{V,E}$
such that $\forall_{i \neq j}: N\brac{v_i} \setminus \set{v_j} = N\brac{v_j} \setminus \set{v_i}$,
and $v_1,\ldots,v_k$ form a clique in $G$.
All graphs $G_{i,j} = G - \brac{v_i,v_j}$ are pairwise isomorphic.
\end{lemma}

\begin{proof}
Consider two graphs $G_{i,j}$ and $G_{i',j'}$ and a function $f: V \rightarrow V$ such that
$f\brac{v_i} = v_{i'}$, $f\brac{v_j} = v_{j'}$. For the other vertices $f$ is an identity.
\end{proof}

A consequence of this lemma is the fact that if a \emph{P}-node contains more than $2$ vertices,
then we may omit all the edges between them, except $\brac{v_1,v_2}$.
The same applies to the vertices that occupy the same sections of a \emph{Q}-node.
Both presented optimizations do not improve the worst case complexity of our algorithm, but significantly speed it up.

\medskip
We implemented the proposed algorithms, and generated all non-isomorphic interval graphs on $n$ vertices for all $n \in \set{1,\ldots,15}$.
The results are available at \url{https://patrykmikos.staff.tcs.uj.edu.pl/graphs/}.

\nocite{*}
\bibliographystyle{abbrvnat}
\bibliography{mikos-dmtcs-episciences}
\label{sec:biblio}

\end{document}